\documentclass{article}

\usepackage{amsmath,amsthm,amsopn,amstext,amscd,amsfonts,amssymb,verbatim}
\usepackage[numbers]{natbib}

\usepackage{graphicx}
\usepackage{verbatim,color}
\usepackage{caption}
\usepackage{subcaption}
\usepackage{tcolorbox}
\usepackage{mathtools}
\usepackage{multirow}

\def\tr{\mathop{\rm tr}\nolimits}
\def\etr{\mathop{\rm etr}\nolimits}
\def\Vec {\mathop{\rm vec}\nolimits}
\def\diag{\mathop{\rm diag}\nolimits}
\def\Cov{\mathop{\rm Cov}\nolimits}
\def\Var{\mathop{\rm Var}\nolimits}
\def\E{\mathop{\rm E}\nolimits}
\def\P{\mathop{\rm P}\nolimits}
\def\build#1#2#3{\mathrel{\mathop{#1}\limits^{#2}_{#3}}}

\newcommand {\boldgreektext}[1] {\boldmath
             \(#1\)\unboldmath}
\newcommand {\boldgreek}[1]% Negritas (modo matem\'{a}tico)
             {\mbox{\boldgreektext{#1}}
            }

\renewenvironment{abstract}
                 {\vspace{6pt}
                  \begin{center}
                  \begin{minipage}{5in}
                  \centerline{\textbf{Abstract}}
                  \noindent\ignorespaces
                 }
                 {\end{minipage}\end{center}}
                 
\newtheorem{theorem}{\textbf{Theorem}}[section]

\theoremstyle{definition}
\newtheorem{definition}{\textbf{Definition}}[section]
\newtheorem{example}{\textbf{Example}}[section]
\newtheorem{remark}{\textbf{Remark}}[section]

\title{\Large \textbf{A generalisation of the chance-constrained Charnes-Cooper approach}}
\author{
  \textbf{Jos\'e A. D\'{\i}az-Garc\'{\i}a} \thanks{Corresponding author\newline
    }\\
  {\normalsize Universidad Aut\'onoma de Chihuahua} \\
  {\normalsize Facultad de Zootecnia y Ecolog\'{\i}a} \\
  {\normalsize Perif\'erico Francisco R. Almada Km 1, Zootecnia} \\
  {\normalsize 33820 Chihuahua, Chihuahua, M\'exico}\\
  {\normalsize E-mail: jadiaz@uach.mx}\\
  \textbf{Francisco J. Caro-Lopera}\\
  {\normalsize University of Medellin} \\
  {\normalsize Faculty of Basic Sciences} \\
  {\normalsize Carrera 87 No.30-65} \\
  {\normalsize Medell\'{\i}n, Colombia} \\
  {\normalsize E-mail: fjcaro@udemedellin.edu.co} \\[2ex]
}
\date{}
\begin{document}
\maketitle

\begin{abstract}
A generalisation of the Charnes-Cooper chance-constrained approach is proposed in the setting of the family of elliptically contoured distributions. The new relaxed stochastic linear programming is notably invariant under the entire class of probability distributions.
\end{abstract}                 

\section{Introduction}\label{sec:1}

Consider the following mathematical programming problem
$$
\begin{array}{c}
  \build{\min}{}{\mathbf{x}} z(\mathbf{x},\boldgreek{\xi}_{0}) \\[2Ex]
  \mbox{subject to:} 
\end{array}
$$
$$
  \begin{array}{c}
    g_{1} (\mathbf{x},\boldgreek{\xi}_{1}) \leq 0 \\
    g_{2} (\mathbf{x},\boldgreek{\xi}_{2}) \leq 0  \\
    \vdots \\
    g_{m} (\mathbf{x},\boldgreek{\xi}_{m}) \leq 0  \\[2Ex]
    x_{j} \geq 0, \ j =1,2,\dots,n, 
  \end{array}
$$
where $z(\cdot)$ is the objective function (commonly termed the cost function in other contexts); $g_i(\cdot)$, $i =\ 1,2,\dots,m$ denote the constraints; $\mathbf{x}=(x_{1} \ \cdots \ x_{n})' \in\Re^{n}$ are the decision variables; and $\boldgreek{\xi}_{i} \in \Re^{k}$, $i=0,1,\dots,m$ are the shape parameters. If $\mathbf{x}$ and/or $\boldgreek{\xi}_{i}$ are random at least for some $i$, then this problem is defined as a \textbf{stochastic programming problem}.

In particular, we are interested in the stochastic linear programming problem, that is, when $z(\cdot)$ and $g_{i}()$ are linear functions. Essentially, there are two approaches for solving a stochastic programming problem. In the first approach, deterministic techniques are extended to the stochastic case; examples of these procedures are: the Ruszczynski feasible directions technique, the subgradient projection method, the extended Lagrange functions, and the min-max procedure, to mention a few, \citet{p:95}, \citet{rc:05}. \citet{bl:11} and \citet{sdr:14}. Another approach consists of methods based on statistical and probabilistic ideas that address the stochastic problem as an equivalent deterministic problem. Specifically, \textbf{a problem is said to be equivalent in the sense that the solution to the new deterministic (linear or nonlinear) programming problem is a solution to the original stochastic programming problem}. Some of the methods developed under this approach are the $E$-Model, $V$-Model, $P$-Model and chance-constrained programming technique, see \citet{chc:59}, \citet{chc:63}, \citet{p:95} and \citet{rc:05}, among many other authors.

In the last approach, the stochastic linear programming problem considers normal distributed random parameters. Moreover, they demand that the parameters of the addressed normal distributions must be known, see \citet{chc:63}. In this context, a more realistic scenario for the stochastic linear programming problem should consider appropriate random samples for parameter estimation. Furthermore, considering that all the parameters of the stochastic linear programming problem follow a normal distribution is very restrictive.

This paper relaxes such restrictions by a generalisation of the Charnes-Cooper chance-constrained method. The new distribution assumptions involves the following gains: each parameter of the stochastic linear programming problem is now considered random; first, it is assumed that there is a random sample, not necessarily of the same size; second, the normality assumption is refused   with a natural flexibly law assumption for each parameter, which can follow an arbitrary elliptically contoured distribution. The new approach provides an equivalent deterministic (linear or nonlinear) programming problem with a remarkable \textbf{\emph{invariant solution}} under any choice of the elliptical distribution assumed for the parameters in the stochastic linear programming problem.

Section \ref{sec:2} summarises some preliminary results necessary for the development of the article. Specifically, the matrix notation and the definition and properties of the elliptically contoured matrix distributions and associated results are established. Then, Section \ref{sec:3} provides a complete exposition of the new approach and the invariant solution of four particular cases that can arise in a stochastic linear programming problem. Each case is also illustrated in the context of a practical situation. 

\section{Preliminaries results}\label{sec:2}

This section recalls some of the well-known notation and definitions of matrix algebra that are needed for the research. In addition, the definition and properties of the vector-elliptically contoured matrix distribution are briefly described; see \citet{fz:90}, \citet{fkn:90}, and \citet{gvb:13} for more details. Finally, the problem of typical interest in stochastic programming is also addressed; some extra information can be seen in  \citet{sm:84}, \citet{up:01} and \citet{sdr:14}.  

\subsection{Notation}

The \textbf{\emph{matrices}} and in particular the \emph{vectors} shall be denoted by bold capital and lowercase letters, respectively. If $\mathbf{A}=(a_{ij})$ is a matrix with $m$ rows and $n$ columns, it is denoted as $\mathbf{A} \in \Re^{m \times n}$. If $m=n$, then $\mathbf{A}$ is termed \textbf{\emph{square matrix}}. If all $a_{ij}=0$, $\mathbf{A}$ is termed \textbf{\emph{zero matrix}}, written as $\mathbf{A}=\mathbf{0}$. If $\mathbf{A}$ is an square matrix such that $a_{ii} = 1$ for all $i = 1,\dots,m$ and $a_{ij} =0$ for all $i\neq j$ then $\mathbf{A}$ is termed \textbf{\emph{identity matrix}}, written $\mathbf{A} = \mathbf{I}_{m}$ or $\mathbf{A} = \mathbf{I}$. The \textbf{\emph{transpose}} of $\mathbf{A} \in \Re^{m \times n}$ is $\mathbf{A}^{'} \in \Re^{n \times m}$, obtained by interchanging the rows and columns of $\mathbf{A}$. Denotes $\mathbf{1} = (1\ 1\ \cdots\ 1)'$. If $\mathbf{A} \in \Re^{m \times m}$ is such that $\mathbf{A} = \mathbf{A}'$ it is said that $\mathbf{A}$ is \textbf{\emph{symmetric matrix}}. A square matrix $\mathbf{A} \in \Re^{m \times m}$ is said to be \textbf{\emph{diagonal matrix}} if all off-diagonal elements are zero, written as $\mathbf{D} \equiv \diag(a_{11}, \dots,a_{mm}) \equiv \diag(a_{ii})$. The \textbf{\emph{determinant}} of a square matrix $\mathbf{A}$, is denoted by $|\mathbf{A}|$. If $|\mathbf{A}| \neq 0$, then $\mathbf{A}$ is termed \textbf{\emph{nonsingular}}, moreover, it exists $\textbf{A}^{-1}$ termed \textbf{\emph{inverse}} of $\mathbf{A}$. $\mathbf{A} \in \Re^{m \times n}$ is partitioned into submatrices if $\mathbf{A}$ can be written as 
$$
  \left(
    \begin{array}{cccc}
      \mathbf{A}_{11} & \mathbf{A}_{12} & \cdots & \mathbf{A}_{1r}\\
      \mathbf{A}_{21} & \mathbf{A}_{21} & \cdots & \mathbf{A}_{1r} \\
      \vdots & \vdots & \ddots & \vdots \\
      \mathbf{A}_{k1} & \mathbf{A}_{k1} & \cdots & \mathbf{A}_{kr} \\
    \end{array}
  \right),
$$
where $\mathbf{A}_{ij} \Re^{m_{i} \times n_{j}}$, such that $i = 1,2, \dots,k$. $j = 1,2,\dots,r$ with $m_{1} + m_{2} + \cdots + m_{k} = m$, and $n_{1} + n_{2} + \cdots + n_{k} = n$.

The \textbf{\emph{trace}} of $\mathbf{A}$ is denoted as $\tr(\mathbf{A})$. It shall be used $e^{x} =\exp(x)$ and $\exp(\tr(\mathbf{A})) = \etr(\mathbf{A})$. If $\mathbf{A}$ is a symmetric matrix, $\mathbf{A} > \mathbf{0}$ denotes a \textbf{\emph{positive definite}} matrix. If $\mathbf{A} = (\mathbf{a}_{1} \ \mathbf{a}_{2}\ \cdots \ \mathbf{a}_{n})$, $\mathbf{a}_{j} \in \Re^{m}$, $j = 1, 2, \dots, n$, $\Vec(\mathbf{A}) \in \Re^{mn}$ denotes the \textbf{\emph{vector}} formed by stacking the columns of $\mathbf{A}$ under each other; that is
$$
  \Vec(\mathbf{A}) = \left(
                       \begin{array}{c}
                         \mathbf{a}_{1} \\
                         \mathbf{a}_{2} \\
                         \vdots \\
                         \mathbf{a}_{n} \\
                       \end{array}
                     \right).
$$ 
The \textbf{\emph{Kronecker product}} of $\mathbf{A} \in \Re^{m \times n}$ and $\mathbf{B} \in \Re^{p \times q}$, is defined and denoted by $\mathbf{A} \otimes \mathbf{B}=(a_{ij}\mathbf{B}) \in \Re^{mp \times nq}$. Finally, $||\mathbf{x}||$ denotes the \textbf{\emph{Frobenious norm}} of $\mathbf{x} \in \Re^{m}$.

\subsection{Vector-elliptically contoured matrix distribution}

The class of elliptically contoured distributions has received increasing attention in statistical literature. See, for example, \citet{fz:90}, \citet{fkn:90}, and \citet{gvb:13}. We shall consider here the nomenclature of \citet{fz:90}.

\begin{definition}\label{def:1}
It is said that a random matrix $\mathbf{X} \in \Re^{p \times q}$ has a \textit{\textbf{vector-elliptically contoured matrix distribution}}, with a \textbf{\emph{location matrix}} $\mathbf{M} \in \Re^{p \times q}$ and a \textbf{\emph{scale matrix}} $\mathbf{C} \otimes \mathbf{D} \in \Re^{pq \times pq}$, $\mathbf{C} \geq {\mathbf{0}}$ and $\mathbf{D} \geq {\mathbf {0}}$, with $\mathbf{C} \in \Re^{p \times p}$ matrix and $\mathbf{D} \in \Re^{q \times q}$, its \textbf{\emph{characteristic function}} is 
$$
  \psi_{\mathbf{X}}(\mathbf{T}) = \etr(i \mathbf{T}^{'}\mathbf{M})\phi\left[\tr(\mathbf{D}\mathbf{T}^{'}\mathbf{C T})\right],
$$
for some function $\phi: [0,\infty)\rightarrow \Re$, $\mathbf{T} \in \Re^{p \times q}$. This distribution shall be denoted as $\mathbf{X} \sim \mathcal{E}_{p \times q}(\mathbf{M}, \mathbf{C} \otimes \mathbf{D}; \phi)$. 
\end{definition}

In addition, if the distribution has a \textbf{\emph{density function}} with respect to Lebesgue measure, this is given by
\begin{equation}\label{eq1}
  dF_{\mathbf{X}}(\mathbf{X}) = |\mathbf{C}|^{-q/2} |\mathbf{D}|^{-p/2} \ g \left\{\tr\left[\mathbf{D}^{-1} (\mathbf{X}-\mathbf{M})^{'} \mathbf{C}^{-1}(\mathbf{X}-\mathbf{M})\right]\right\}(d\mathbf{X}),
\end{equation}
where $\mathbf{C} > {\mathbf{0}}$ and $\mathbf{D} > {\mathbf {0}}$ and $(d\mathbf{X})$ denotes the Lebesgue measure on $\Re^{p \times q}$. The function $g: \Re \rightarrow [0, \infty)$ is termed the \textbf{density generator}, such that $\int^{\infty}_{0} v^{pq/2-1} g(v) \ dv < \infty$; where $g$ and $\phi$ determine each other for specified $p$ and $q$. This fact shall be denoted as $\mathbf{X} \sim \mathcal{E}_{p \times q}(\mathbf{M}, \mathbf{C} \otimes \mathbf{D}; g)$.  In addition, from \citet[p. 59]{fz:90} we have that 
\begin{equation}\label{eq:g}
  1 = \int_{\mathbf{X}}dF(\tr(\mathbf{X}'\mathbf{X}))(d\mathbf{X}) = \frac{\pi^{pq/2}}{\Gamma[pq/2]}\int^{\infty}_{0} v^{pq/2-1} g(v) \ (dv) 
\end{equation}

This class of distributions includes the matrix variate Normal, $T-$, Contaminated Normal, Logist Pearson type II and VII and Power Exponential, among other distributions; whose shapes have tails that are more or less weighted and/or present a greater or smaller degree of kurtosis than the matrix variate normal distribution.

Now, the matrix variate normal case of \citet[Theorem 3.1.1, p.79]{mh:05} in terms of the elliptical setting of \citet[Theorem 2.1, p.16]{gvb:13} provides that $\Vec(\mathbf{X}) \sim \mathcal{E}_{pq}(\Vec(\mathbf{M}), \mathbf{C} \otimes \mathbf{D}; \phi)$. Then, we can propose an alternative definition of the vector-elliptically contoured matrix distribution. Specifically, this idea is summarised as $\Vec(\mathbf{X}) \sim \mathcal{E}_{pq}(\boldgreek{\delta}, \boldgreek{\Xi};\phi)$, where $\boldgreek{\delta} = \Vec(\mathbf{M}) \in \Re^{pq}$ and $\boldgreek{\Xi} \in \Re^{pq \times pq}$ are arbitrary. Observe that if $\mathbf{X}$ has \textbf{\emph{moments}}, hence  $\E(\Vec(\mathbf{X})) = \boldgreek{\delta} \in \Re^{pq}$ and  $\Cov(\Vec(\mathbf{X}')) =-2 \phi'(0)\boldgreek{\Xi} \in \Re^{pq \times pq}$, where 
$$
 \phi'(0) = \left.\frac{\partial \phi(\Vec(\mathbf{T}))}{\partial \Vec(\mathbf{T})}\right|_{\mathbf{T} =0}.
$$ 
This generalisation in the parameter definition is termed parameter enrichment and is frequently used in Bayesian analysis, see \citet[Section 8.6,2, pp. 252-254]{p:82}. Thus, $\Vec(\mathbf{X})$ has a density function with respect to Lebesgue measure $(d\Vec(\mathbf{X}))$ on $\Re^{pq}$ and it is given by 
\begin{equation}\label{eq2}
  dF_{\Vec(\mathbf{X})}(\Vec(\mathbf{X})) = |\boldgreek{\Xi}|^{-1/2} \ g \left\{\left[\Vec(\mathbf{X})-\boldgreek{\delta}\right]^{'} \boldgreek{\Xi}^{-1} \left[\Vec(\mathbf{X})-\boldgreek{\delta}\right]\right\} (d\Vec(\mathbf{X})),
\end{equation}
where $\Vec^{'}(\mathbf{X}) = (\Vec(\mathbf{X}))'$.

Then, by \citet[eq. 3.3.10, p. 103]{fz:90} and \citet[Theorem 2.2, p. 16]{gvb:13}, we have the following result:
\begin{theorem}\label{teo:1}
  Suppose that $\mathbf{E} \in \Re^{r \times p}$, $\mathbf{F} \in \Re^{q \times s}$ and $\mathbf{C} \in \Re^{r \times s}$ are constant matrices. Let $\mathbf{X} \sim \mathcal{E}_{r \times s}(\mathbf{M}, \mathbf{C} \otimes \mathbf{D}; \phi)$, then
  $$
    \mathbf{EXF} + \mathbf{C} \sim \mathcal{E}_{r \times s}(\mathbf{EM}\mathbf{F} + \mathbf{C}, (\mathbf{E \Sigma E'}) \otimes (\mathbf{F'\Theta F}); \phi).
  $$  
\end{theorem}

Moreover, by \citet[Corollary 1, p. 135]{fz:90} we have that:
   
\begin{theorem}\label{teo2}
  Let $\mathbf{x}_{j} \in \Re^{q}$, $j = 1,2,\dots,p$, be a sample of random vectors such that
  $$
      \mathbf{X} = \left(
           \begin{array}{c}
             \mathbf{x}'_{1} \\
             \mathbf{x}'_{2} \\
             \vdots \\
             \mathbf{x}'_{p} \\
           \end{array}
         \right) \sim \mathcal{E}_{p \times q}(\mathbf{1}\boldgreek{\mu}', \mathbf{I}_{p} \otimes \mathbf{\Sigma}; g).
  $$   
 Then, the joint density of  $\overline{\mathbf{x}}$ and $\mathbf{S}$, $dF_{\overline{\mathbf{x}}, \mathbf{S}}(\overline{\mathbf{x}}, \mathbf{S})$ is
  $$
     = \frac{p^{q/2}\pi^{(p-1)q/2}}{\Gamma_{q}[(p-1)/2]} \frac{|\mathbf{S}|^{(p-q)/2-1}}{|\mathbf{\Sigma}|^{p/2}} g\left[\tr \mathbf{\Sigma}^{-1}\mathbf{S} + p(\overline{\mathbf{x}}-\boldgreek{\mu})'\mathbf{\Sigma}^{-1}(\overline{\mathbf{x}}-\boldgreek{\mu})\right] (d\overline{\mathbf{x}})\wedge (d\mathbf{S}),
  $$
  where $p > q$, $\overline{\mathbf{x}} \in \Re^{q}$, $\mathbf{S} \in \Re^{q \times q}$, $\mathbf{S} > \mathbf{0}$ and $(d\overline{\mathbf{x}})$ and $(d\mathbf{S})$ are the Lebesgue measures on $\Re^{q}$ and the cone of positive definite matrices on $\Re^{q \times q}$, respectively. Meanwhile,
  $$
    \overline{\mathbf{x}} = \displaystyle\frac{1}{p} \sum_{k =1}^{p}  \mathbf{x}_{k},
  $$
  and
  $$
    \mathbf{S} = \displaystyle\sum_{k =1}^{p} (\mathbf{x}_{k}-\overline{\mathbf{x}})(\mathbf{x}_{k}-\overline{\mathbf{x}})' = \mathbf{X}'\left(\mathbf{I}_{p} - \frac{1}{p} \mathbf{1}_{p} \mathbf{1}'_{p}\right)\mathbf{X}.
  $$
\end{theorem}
Finally, a transcendental result for this work is set as follows:
\begin{theorem}\label{teo3}
  In Theorem \ref{teo2} take $q=1$ and define the random variable 
  $$
    T = \frac{\displaystyle\frac{(\overline{x}-\mu)}{\sqrt{\displaystyle\frac{\sigma^{2}}{p}}}}{\sqrt{\displaystyle\frac{s}{(p-1)\sigma^{2}}}} = \frac{\sqrt{p}\sqrt{p-1}(\overline{x}-\mu)}{\sqrt{s}}.
  $$
  Then, $T$ follows a t-student distribution with $p-1$ degree of freedom; a fact denoted as $T \sim \mbox{t}_{p-1}$.  Thus, the random variable $T$ has the same distribution for the entire family of elliptically contoured distributions.
\end{theorem} 
\begin{proof}
  The demonstration of this invariant result can be obtained as a consequence of \citet[Theorem 5.1.1, p.154]{fz:90} and \citet[Theorem 5.12, p. 139]{gvb:13}. Alternatively, taking $q =1$ in Theorem \ref{sec:2} and denoting $\Gamma_{1}[\cdot] \rightarrow \Gamma[\cdot]$, $\mathbf{\Sigma \rightarrow \sigma^{2}}$, $\mathbf{S} \rightarrow s$, $\boldgreek{\mu} \rightarrow \mu$ and $\overline{\mathbf{x}} \rightarrow \overline{x}$, then the join density of $\overline{X}$ and $S$ is 
  $$
     dF_{(\overline{X},S)}(\overline{x},s)= \frac{p^{1/2}\pi^{(p-1)/2} s^{(p-1)/2-1}}{\Gamma[(p-1)/2](\sigma^{2})^{p/2}}  g\left[\frac{1}{\sigma^{2}}\left(s + p(\overline{x}-\mu)^{2}\right)\right] (d\overline{x})\wedge (ds).
  $$
  Making the change of variable
  $$
    t = \frac{\sqrt{p}\sqrt{p-1}(\overline{x}-\mu)}{\sqrt{s}} \quad \mbox{and}\quad s=r,
  $$
  hence  
  $$
    \overline{x} = \frac{\sqrt{r} t}{\sqrt{p}\sqrt{p-1}} +\mu \quad \mbox{and}\quad r = s,
  $$
  noting that
  $$
    |J[(\overline{x},s)\rightarrow (t,r)]| = \left |
            \begin{array}{cc}
              \displaystyle\frac{\partial \overline{x}}{\partial t} & \displaystyle\frac{\partial \overline{x}}{\partial r} \\[2ex]
              \displaystyle\frac{\partial s}{\partial t} & \displaystyle\frac{\partial s}{\partial r} 
            \end{array}
            \right | =
            \left|
              \begin{array}{cc}
                \displaystyle\frac{\sqrt{r}}{\sqrt{p}\sqrt{p-1}} & \displaystyle\frac{t}{2\sqrt{p}\sqrt{p-1}\sqrt{r}} \\[2ex]
                0 & 1 \\
              \end{array}
            \right|. 
  $$
  Therefore
  $$
    (d\overline{x})\wedge (ds) = \frac{\sqrt{r}}{\sqrt{p}\sqrt{p-1}} (dt)\wedge (dr).
  $$
  Thus,
  $$
     dF_{(T,R)}(t,r)= \frac{\pi^{(p-1)/2} r^{(p-1)/2-1}}{\Gamma[(p-1)/2](\sigma^{2})^{p/2}\sqrt{p-1}}  g\left[\frac{r}{\sigma^{2}}\left(1 + \frac{t^{2}}{(p-1)}\right)\right] (dt)\wedge(dr).
  $$
  Now, making the change of variable 
  $$
    w= \frac{r}{\sigma^{2}}\left(1 + \frac{t^{2}}{(p-1)}\right),
  $$
  we have that 
  $$
     (dr) = \sigma^{2}\left(1 + \frac{t^{2}}{(p-1)}\right)^{-1}(dw).
  $$ 
  Thus, the joint density of $t$ and $w$ is
  $$
     dF_{(T,W)}(t,w)= \frac{\pi^{(p-1)/2} }{\Gamma[(p-1)/2]\sqrt{p-1}} \left(1 + \frac{t^{2}}{(p-1)}\right)^{-p/2}  w^{p/2-1}g(w) (dw)\wedge(dt).
  $$ 
  Applying equation (\ref{eq:g})
  $$
    \int_{0}^{\infty}  w^{p/2-1}g(w) (dw) = \frac{\Gamma[p/2]}{\pi^{p/2}},
  $$
  we finally obtain the required marginal density of $T$:
  $$
    dF_{T}(t)= \frac{\Gamma[p/2]}{\Gamma[(p-1)/2]\sqrt{\pi(p-1)}} \left(1 + \frac{t^{2}}{(p-1)}\right)^{-p/2} (dt),
  $$
  which means that $T \sim t_{p-1}$.
\end{proof}
Now, we are in position to propose the new generalised and invariant chance-constrained programming method.
\section{Modified chance-constrained programming technique}\label{sec:3}

The \emph{chance-constrained programming problem} to be solved takes the form:
\begin{equation}\label{eqccp}
\begin{array}{c}
  \build{\min}{}{\mathbf{x}} z(\mathbf{x}) = \mathbf{c}'\mathbf{x} \\[2Ex]
  \mbox{subject to:}\\[2ex] 
  \P(\mathbf{a}'_{i}\mathbf{x} \leq b_{i}) \geq (1-\alpha_{i}), \quad i =1,2,\dots,m,  \\[2ex]
  x_{j} \geq 0, \ j =1,2,\dots,n, 
  \end{array}
\end{equation}
where $\mathbf{x} = (x_{1} \ x_{2} \cdots x_{n})' \in \Re^{n}$ (are the decision variables), $\P(\cdot)$ denotes the probability and $\alpha_{i} \in\ (0,1), \ i = 1,2,\dots,m$, $\mathbf{c} = (c_{1} \ c_{2} \cdots c_{n})' \in \Re^{n}$, $\mathbf{a}_{i} \in \Re^{n}$, $i = 1,2,\dots,m$, $\mathbf{A}' = (\mathbf{a}_{1} \ \mathbf{a}_{2} \cdots \mathbf{a}_{m})$, $\mathbf{A} \in \Re^{m \times n}$ and $\mathbf{b} = (b_{1} \ b_{2} \cdots b_{n})' \in \Re^{m}$. In this setting, there are four situations of interest. The first 3 cases consider only one isolated random class, namely: 1) $\mathbf{c}$ is random; 2) $\mathbf{A}$ is random; 3) $\mathbf{b}$ is random. The final case allows simultaneous randomness in $\mathbf{c}$, $\mathbf{A}$, and $\mathbf{b}$. In addition, \citet{chc:59} assume that $\mathbf{c}$, $\mathbf{a}_{i}, \ i=1,2,\dots,m,$ and $\mathbf{b}$ have a multivariate normal distribution, such that their parameters (expectation and covariance matrix) are known. Then, with these assumptions, \citet{chc:59} propose the corresponding equivalent deterministic problems.

Now, our invariant new approach sets that 
\begin{itemize}
  \item $\mathbf{c} \sim \mathcal{E}_{n}(\boldgreek{\mu}_{\mathbf{c}}, \mathbf{\Sigma}_{\mathbf{c}}; g_{\mathbf{c}})$, with
     $$
       \boldgreek{\mu}_{\mathbf{c}} =
           \left(
             \begin{array}{c}
               \mu_{c_{1}} \\
               \mu_{c_{2}} \\
               \vdots \\
               \mu_{c_{n}} \\
             \end{array}
           \right), \quad 
           \mathbf{\Sigma}_{\mathbf{c}} =
             \left(
             \begin{array}{cccc}
                \sigma^{2}_{c_{1}} & \sigma_{c_{1},c_{2}} & \cdots & \sigma_{c_{1},c_{n}} \\
                \sigma_{c_{2},c_{1}} & \sigma^{2}_{c_{2}} & \cdots & \sigma_{c_{2},c_{n}} \\
                \vdots & \vdots & \ddots & \vdots \\
                \sigma_{c_{n},c_{1}} & \sigma_{c_{n},c_{2}} & \cdots & \sigma^{2}_{c_{n}} \\
             \end{array}
             \right),
     $$
     where $\E(c_{i}) = \mu_{c_{i}}$, $\Var(c_{i}) = \sigma^{2}_{c_{i}}$ and $\Cov(c_{i},c_{j}) = \sigma_{c_{i},c_{j}}$, $i, j = 1,2,\dots,n$.
  \item $\Vec(\mathbf{A})' \sim \mathcal{E}_{nm}(\Vec(\boldgreek{\mu}'_{\mathbf{A}}), \mathbf{\Sigma}_{\mathbf{A}}, g_{\mathbf{A}}),$ or $\mathbf{a}_{i} \sim \mathcal{E}_{n}(\boldgreek{\mu}_{\mathbf{a}_{i}}, \mathbf{\Sigma}_{\mathbf{a}_{i}}; g_{\mathbf{a}_{i}})$, \newline $i =1,2,\dots,m$, uncorrelated, where
     $$
       \boldgreek{\mu}_{\mathbf{a}_{i}} =
           \left(
             \begin{array}{c}
               \mu_{a_{i1}} \\
               \mu_{a_{i2}} \\
               \vdots \\
               \mu_{a_{in}} \\
             \end{array}
           \right), \quad 
           \mathbf{\Sigma}_{\mathbf{a}_{i}} =
             \left(
             \begin{array}{cccc}
                \sigma^{2}_{a_{i1}} & \sigma_{a_{i1},a_{i2}} & \cdots & \sigma_{a_{i1},a_{in}} \\
                \sigma_{a_{i2},a_{i1}} & \sigma^{2}_{a_{i2}} & \cdots & \sigma_{a_{i2},a_{in}} \\
                \vdots & \vdots & \ddots & \vdots \\
                \sigma_{a_{in},a_{i1}} & \sigma_{a_{in},a_{i2}} & \cdots & \sigma^{2}_{a_{in}} \\
             \end{array}
             \right),
     $$
     with $\E(a_{ij}) = \mu_{a_{ij}}$, $\Var(a_{ij}) = \sigma^{2}_{a_{ij}}$ and $\Cov(a_{ij},a_{ik}) = \sigma_{a_{ij},a_{ik}}$, $i, j, k = 1,2,\dots,n$.
     And denoting $ \mathbf{\Sigma}_{\mathbf{A}} = \Cov(\Vec(\mathbf{A}'))$
     $$
       \Vec(\boldgreek{\mu}'_{\mathbf{A}}) = 
                                        \left(
                                          \begin{array}{c}
                                            \boldgreek{\mu}_{\mathbf{a}_{1}} \\
                                            \boldgreek{\mu}_{\mathbf{a}_{2}}  \\
                                            \vdots \\
                                            \boldgreek{\mu}_{\mathbf{a}_{m}}  \\
                                          \end{array}
                                        \right) \qquad
      \mathbf{\Sigma}_{\mathbf{A}} = \left(
                                         \begin{array}{cccc}
                                             \mathbf{\Sigma}_{\mathbf{a}_{1}} & \mathbf{0} & \cdots & \mathbf{0} \\
                                             \mathbf{0} & \mathbf{\Sigma}_{\mathbf{a}_{2}} & \cdots & \mathbf{0} \\
                                             \vdots & \vdots & \ddots & \vdots \\
                                             \mathbf{0} & \mathbf{0} & \cdots & \mathbf{\Sigma}_{\mathbf{a}_{m}} \\
                                         \end{array}
                                      \right),
     $$ 
     and  
  \item $\mathbf{b} \sim \mathcal{E}_{n}(\boldgreek{\mu}_{\mathbf{b}}, \mathbf{\Sigma}_{\mathbf{b}}; g_{\mathbf{b}})$, where 
     $$
       \boldgreek{\mu}_{\mathbf{b}} =
           \left(
             \begin{array}{c}
               \mu_{b_{1}} \\
               \mu_{b_{2}} \\
               \vdots \\
               \mu_{b_{n}} \\
             \end{array}
           \right), \quad 
           \mathbf{\Sigma}_{\mathbf{b}} =
             \left(
             \begin{array}{cccc}
                \sigma^{2}_{b_{1}} & \sigma_{b_{1},b_{2}} & \cdots & \sigma_{b_{1},b_{n}} \\
                \sigma_{b_{2},b_{1}} & \sigma^{2}_{b_{2}} & \cdots & \sigma_{b_{2},b_{n}} \\
                \vdots & \vdots & \ddots & \vdots \\
                \sigma_{b_{n},b_{1}} & \sigma_{b_{n},b_{2}} & \cdots & \sigma^{2}_{b_{n}} \\
             \end{array}
             \right),
     $$
     and $\E(b_{i}) = \mu_{b_{i}}$, $\Var(b_{i}) = \sigma^{2}_{b_{i}}$ and $\Cov(b_{i},b_{j}) = \sigma_{b_{i},b_{j}}$, $i, j = 1,2,\dots,n$. 
\end{itemize}
Furthermore, the following random samples exist:
\begin{itemize}
  \item Let $\mathbf{c}_{1}, \mathbf{c}_{2}, \dots, \mathbf{c}_{N_{\mathbf{c}}}$ be a random sample such that if $\mathbf{C}' =(\mathbf{c}_{1}\ \mathbf{c}_{2}\ \cdots\ \mathbf{c}_{N_{\mathbf{c}}})$, then
    $$
      \mathbf{C} = \left(
           \begin{array}{c}
             \mathbf{c}'_{1} \\
             \mathbf{c}'_{2} \\
             \vdots \\
             \mathbf{c}'_{N_{\mathbf{c}}} \\
           \end{array}
         \right) \sim \mathcal{E}_{N_{\mathbf{c}}          
         \times n}(\mathbf{1}_{N_{\mathbf{c}}}\boldgreek{\mu}_{\mathbf{c}}',\mathbf{I}_{N_{\mathbf{c}}} \otimes \mathbf{\Sigma}_{\mathbf{c}}; g_{\mathbf{c}}),
    $$
  \item Also, let $\mathbf{a}_{i_{1}}, \mathbf{a}_{i_{2}}, \dots, \mathbf{a}_{i_{N_{\mathbf{a}_{i}}}}$ be a random samples such that
    $$
      \mathbf{A}_{i} = \left(
           \begin{array}{c}
             \mathbf{a}'_{i_{1}} \\
             \mathbf{a}'_{i_{2}} \\
             \vdots \\
             \mathbf{a}'_{i_{N_{\mathbf{a}_{i}}}} \\
           \end{array}
         \right) 
         \sim \mathcal{E}_{N_{\mathbf{a}_{i}}       
         \times n} (\mathbf{1}_{N_{\mathbf{a}_{i}}}\boldgreek{\mu}'_{\mathbf{a}_{i}}, \mathbf{I}_{N_{\mathbf{a}_{i}}} \otimes \mathbf{\Sigma}_{\mathbf{a}_{i}}; g_{\mathbf{a}_{i}}), \ i = 1, 2, \dots, m,
    $$ 
    and
  \item Let $\mathbf{b}_{1}, \mathbf{b}_{2}, \dots, \mathbf{b}_{N_{\mathbf{b}}}$ be a random sample whose elements are the rows of the random matrix $\mathbf{B}$, such that  
       $$
      \mathbf{B} = \left(
           \begin{array}{c}
             \mathbf{b}'_{1} \\
             \mathbf{b}'_{2} \\
             \vdots \\
             \mathbf{b}'_{N_{\mathbf{c}}} \\
           \end{array}
         \right) \sim \mathcal{E}_{N_{\mathbf{b}}          
         \times n}(\mathbf{1}_{N_{\mathbf{b}}}\boldgreek{\mu}_{\mathbf{b}}',\mathbf{I}_{N_{\mathbf{b}}} \otimes \mathbf{\Sigma}_{\mathbf{b}}; g_{\mathbf{b}}).
    $$
\end{itemize}
In terms of these random samples the following \textbf{\emph{maximum likelihood estimators}} $(\widetilde{\boldgreek{\mu}}, \widetilde{\mathbf{\Sigma}})$  of the parameters $(\boldgreek{\mu}, \mathbf{\Sigma})$ are obtained, see \citet[Section 4.1.1, pp.127-130]{fz:90} and \citet[Section 7.1, p. 173]{gvb:13}, for $N_{\mathbf{c}} \geq n$, $N_{\mathbf{a}_{i}} \geq n, \ i =1,2,\dots,m,$ and $N_{\mathbf{b}} \geq n$:
\begin{itemize}
    \item $\left(\widetilde{\boldgreek{\mu}}_{\mathbf{c}}, \widetilde{\mathbf{\Sigma}}_{\mathbf{c}}\right) = \left(\overline{\mathbf{c}}, \lambda_{max}(h_{\mathbf{c}}) \ \mathbf{S}_{\mathbf{c}} \right)$,
    \item $\left(\widetilde{\boldgreek{\mu}}_{\mathbf{a}_{i}}, \widetilde{\mathbf{\Sigma}}_{\mathbf{a}_{i}}\right) = \left(\overline{\mathbf{a}}_{i}, \lambda_{max}(h_{\mathbf{a}_{i}}) \ \mathbf{S}_{\mathbf{a}_{i}} \right), \ i = 1,2,\dots,m,$ and 
    \item $\left(\widetilde{\boldgreek{\mu}}_{\mathbf{b}}, \widetilde{\mathbf{\Sigma}}_{\mathbf{b}}\right) = \left(\overline{\mathbf{b}}, \lambda_{max}(h_{\mathbf{b}}) \ \mathbf{S}_{\mathbf{b}} \right)$,
  \end{itemize}
where $\lambda_{max}(\cdot)$ is defined in \citet[Lemma 4.1.2, p. 129]{fz:90} as the maximum of the function
$$
  f(\lambda) = \lambda^{-pn/2} g(n/\lambda);
$$
where $g(\cdot)$ is respectively given by $g_{\mathbf{c}}(\cdot)$, $g_{\mathbf{b}}(\cdot)$ and $g_{\mathbf{a}_{i}}(\cdot), \ \ i = 1,2,\dots, m$, and $p = N_{\mathbf{c}},\ N_{\mathbf{\mathbf{a}}_{i}}, \ i = 1,2,\dots, m, \ N_{\mathbf{b}}$.

Also, the following \textbf{\emph{sample means}} ($\overline{\mathbf{c}}$, $\overline{\mathbf{b}}$ and $\overline{\mathbf{a}_{i}}, \ i = 1,2,\dots,m$) and \textbf{\emph{sample covariance matrices}} $\left(\frac{1}{N_{\mathbf{c}}}\mathbf{S}_{\mathbf{c}}, \ \frac{1}{N_{\mathbf{b}}}\mathbf{S}_{\mathbf{b}}, \ \frac{1}{N_{\mathbf{a}_{i}}}\mathbf{S}_{\mathbf{a}_{i}}, \ i =1,2,\dots,m\right)$ are given:
\begin{itemize}
  \item $$
          \overline{\mathbf{c}} = \displaystyle\frac{1}{N_{\mathbf{c}}} \sum_{k =1}^{N_{\mathbf{c}}}  \mathbf{c}_{k},
        $$
        and
        $$
          \mathbf{S}_{\mathbf{c}} = \displaystyle\sum_{k =1}^{N_{\mathbf{c}}} (\mathbf{c}_{k}-\overline{\mathbf{c}})(\mathbf{c}_{k}-\overline{\mathbf{c}})' = \mathbf{C}'\left(\mathbf{I}_{N_{\mathbf{c}}} - \frac{1}{N_{\mathbf{c}}} \mathbf{1}_{N_{\mathbf{c}}} \mathbf{1}'_{N_{\mathbf{c}}}\right)\mathbf{C}.
        $$
  \item For $i =1,2,\dots,m$ we have that
        $$
          \overline{\mathbf{a}}_{i} = \frac{1}{N_{\mathbf{a}_{i}}} \displaystyle\sum_{k =1}^{N_{\mathbf{a}_{i}}} \mathbf{a}_{i_{k}},
        $$
        and
        $$ 
          \mathbf{S}_{\mathbf{a}_{i}} = \displaystyle\sum_{k =1}^{N_{\mathbf{a}_{i}}} (\mathbf{a}_{i_{k}} -\overline{\mathbf{a}}_{i})(\mathbf{a}_{i_{k}}-\overline{\mathbf{a}}_{i})'  = \mathbf{A}'_{i}\left(\mathbf{I}_{N_{\mathbf{a}_{i}}} - \frac{1}{N_{\mathbf{a}_{i}}} \mathbf{1}_{N_{\mathbf{a}_{i}}} \mathbf{1}'_{N_{\mathbf{a}_{i}}}\right)\mathbf{A}_{i}; 
        $$ 
        similarly,
  \item $$
         \overline{\mathbf{b}} = \displaystyle \frac{1}{N_{\mathbf{b}}} \sum_{k =1}^{N_{\mathbf{b}}} \mathbf{c}_{k},
        $$ 
        $$
          \mathbf{S}_{\mathbf{b}} = \displaystyle\sum_{k =1}^{N_{\mathbf{b}}} (\mathbf{b}_{k}-\overline{\mathbf{b}})(\mathbf{b}_{k}-\overline{\mathbf{b}})' = \mathbf{B}'\left(\mathbf{I}_{N_{\mathbf{b}}} - \frac{1}{N_{\mathbf{b}}} \mathbf{1}_{N_{\mathbf{b}}} \mathbf{1}'_{N_{\mathbf{b}}}\right)\mathbf{B}.
        $$
\end{itemize}

Thus, \citet[Lemma 4.3.1, p. 138]{fz:90} provides the following unbiased estimators: 
\begin{itemize}
    \item $\left(\widehat{\boldgreek{\mu}}_{\mathbf{c}}, \widehat{\mathbf{\Sigma}}_{\mathbf{c}}\right) = \left(\overline{\mathbf{c}}, \displaystyle\frac{1}{2(1-N_{\mathbf{c}})\phi_{\mathbf{c}}'(0)}\ \mathbf{S}_{\mathbf{c}} \right)$,
    \item $\left(\widehat{\boldgreek{\mu}}_{\mathbf{a}_{i}}, \widehat{\mathbf{\Sigma}}_{\mathbf{a}_{i}}\right) = \left(\overline{\mathbf{a}}_{i}, \displaystyle\frac{1}{2(1-N_{\mathbf{a}_{i}})\phi_{\mathbf{a}_{i}}'(0)} \ \mathbf{S}_{\mathbf{a}_{i}} \right), \ i = 1,2,\dots,m,$ and 
    \item $\left(\widehat{\boldgreek{\mu}}_{\mathbf{b}}, \widehat{\mathbf{\Sigma}}_{\mathbf{b}}\right) = \left(\overline{\mathbf{b}}, \displaystyle\frac{1}{2(1-N_{\mathbf{b}})\phi_{\mathbf{b}}'(0)} \ \mathbf{S}_{\mathbf{b}} \right)$.
  \end{itemize}
  
In addition, Theorems \ref{teo2} and \ref{teo3} provide the following stochastic linear programming scenarios.

\subsection{Case I. The vector of coefficients $\mathbf{c}\in \Re^{n}$ is random.}

In this case the vectors $\mathbf{b}$ and $\mathbf{a}_{i}, \ i = 1,2, \dots,m,$ are non random. Hence, the chance-constrained programming problem is
\begin{equation}\label{eqcp}
\begin{array}{c}
  \build{\min}{}{\mathbf{x}} z(\mathbf{x}) = \overline{\mathbf{c}}'\mathbf{x} \\[2Ex]
  \mbox{subject to:}\\[2ex] 
  \mathbf{a}'_{i}\mathbf{x} \leq b_{i}, \quad i =1,2,\dots,m,  \\[2ex]
  x_{j} \geq 0, \ j =1,2,\dots,n. 
  \end{array}
\end{equation}
Hence, by Theorem \ref{teo:1}, $\overline{\mathbf{c}} \sim \mathcal{E}_{n}\left(\boldgreek{\mu}_{\mathbf{c}}, \frac{1}{N_{\mathbf{c}}}\mathbf{\Sigma}_{\mathbf{c}}; g_{\mathbf{c}}\right)$, and
$$ 
  z(\mathbf{x}) =\overline{\mathbf{c}}'\mathbf{x} \sim \mathcal{E}_{1}\left(\boldgreek{\mu}_{\mathbf{c}}'\mathbf{x}, \frac{1}{N_{\mathbf{c}}}\mathbf{x}'\mathbf{\Sigma}_{\mathbf{c}}\mathbf{x}; g_{\mathbf{c}}\right) \equiv \mathcal{E}_{1}(\E(z(\mathbf{x})), \Var(z(\mathbf{x}));g_{\mathbf{c}}),
$$
where $\E(z(\mathbf{x})) = \boldgreek{\mu}_{\mathbf{c}}'\mathbf{x}$ and $\Var(z(\mathbf{x})) = N_{\mathbf{c}}^{-1}\mathbf{x}'\mathbf{\Sigma}_{\mathbf{c}}\mathbf{x}$. Thus, the emergent \textbf{\emph{equivalent deterministic nonlinear programming problem corresponding to the stochastic programming problem}} (\ref{eqcp}) is stated as 
\begin{equation}\label{eqcpE}
\begin{array}{c}
  \build{\min}{}{\mathbf{x}} Z(\mathbf{x}) = k_{1}\overline{\mathbf{c}}'\mathbf{x} + k_{2} \displaystyle\sqrt{ \widetilde{\Var}(\overline{\mathbf{c}}'\mathbf{x})}\\[2Ex]
  \mbox{subject to:}\\[2ex] 
  \mathbf{a}'_{i}\mathbf{x} \leq b_{i}, \quad i =1,2,\dots,m,  \\[2ex]
  x_{j} \geq 0, \ j =1,2,\dots,n, 
  \end{array}
\end{equation}
where $\widetilde{\Var}(\overline{\mathbf{c}}'\mathbf{x})$ denotes an estimated variance of $\overline{\mathbf{c}}'\mathbf{x}$; and it is a function of $g_{\mathbf{c}}(\cdot)$, if the maximum likelihood or unbiased estimators are used. However, for practical purposes, the unbiased sample covariance matrix ($\mathbf{S}^{*}_{\mathbf{c}} = (N_{\mathbf{c}}-1)^{-1} \mathbf{S}_{\mathbf{c}}$) can be used as an estimator, which is independent of the generating function $g_{\mathbf{c}}(\cdot)$. In this non-indexed kernel case, it just takes the form $\widetilde{\Var}(\overline{\mathbf{c}}'\mathbf{x}) = N_{\mathbf{c}}^{-1}\mathbf{x}'\mathbf{S}^{*}_{\mathbf{c}}\mathbf{x}$. In addition, observe that $k_{1}$ and $k_{2}$ are constants whose values determine the relative importance of $\overline{\mathbf{c}}'\mathbf{x}$ and $\sqrt{N_{\mathbf{c}}^{-1}\mathbf{x}'\mathbf{S}^{*}_{\mathbf{c}}\mathbf{x}}$ in the new objective function $Z(\cdot)$. The constants $k_{1}$ and  $k_{2}$ are such that
$$
  k_{1} \geq 0, \ k_{2} \geq 0 \mbox{ and } k_{1} + k_{2} = 1.
$$
Now, $k_{1} = 0$ privileges the minimization of the variability of $Z(\cdot)$ around its mean value, without considering the mean value of $Z(\cdot)$. Meanwhile, $k_{2} = 0$ indicates that the mean value of $Z(\cdot)$ should be minimized regardless of the variability of $Z(\cdot)$. Similarly, $k_{1}=k_{2}$ reflects an equal importance for the mean value and the standard deviation. Other complementary values weight the corresponding importance for the mean and the variability.
 
If the program pursues a maximization problem, the objective function shall maximize the mean value of the $Z\left(\mathbf{x}\right)$ and minimize the standard deviation, a fact that requires a negative sign for $k_{2}$. This is concluded by recalling that, $\build{\max}{}{\mathbf{x}} f(\mathbf{x})\ = - \build{\min}{}{\mathbf{x}} f(\mathbf{x})$.

When all the random variables $c_{j}$ are non correlated, the objective function reduces to 
$$
  Z(\mathbf{x}) = k_{1}\sum_{j=1}^{n} \overline{c}_{j}x_{j} +k_{2}\sqrt{\sum_{j=1}^{n}\widetilde{\Var}(\overline{c}_{j})x_{j}^{2}}.
$$

\begin{remark}
Observe that the solution (\ref{eqcpE}) also is a solution of the following \textbf{\emph{multi-objective programming problem}}
\begin{equation}\label{meqcpe}
\begin{array}{c}
  \build{\min}{}{\mathbf{x}} 
   \left(
   \begin{array}{c}
     \overline{\mathbf{c}}'\mathbf{x} \\[2ex]
     \displaystyle\sqrt{ \widetilde{\Var}(\overline{\mathbf{c}}'\mathbf{x})} 
   \end{array}
   \right )\\[2ex]
  \mbox{subject to:}\\[2ex] 
  \mathbf{a}'_{i}\mathbf{x} \leq b_{i}, \quad i =1,2,\dots,m,  \\[2ex]
  x_{j} \geq 0, \ j =1,2,\dots,n. 
  \end{array}
\end{equation}
In this case the solution is achieved based in the concept of sum value function proposed by \citet{z:63}. In the context of multi-objective programming, the \textbf{\emph{sum value function}} provides a sufficient condition for \textbf{\emph{Pareto optimality}}; thus the minimum of the sum function value is always Pareto optimality. This observation opens up a set of potential alternative solutions to problem (\ref{eqcp}), in terms of all possible solutions to multi-objective problem (\ref{meqcpe}); for more details see \citet{dgrc:05}.  
\end{remark}

The equivalent deterministic statement for stochastic optimization problems allows the use of several softwares initially thought for classical optimization. 

\begin{example}\label{ex1}

A manufacturing company produces three components of machinery using lathes, milling machines, and grinding machines. Table \ref{tab1} shows the machining times available per week on different machines, the profitability of each part and the machining times required on the different machines for each manufactured component. The goal is to determine the number of machine components I, II and III that should be manufactured per week to maximize profit. 

\medskip
\begin{table}[h!] 
    \centering 
    \caption{Data Example 1} 
    \label{tab1} 
\begin{scriptsize}
\begin{tabular}{l|rrrr}
\hline\hline
  \multirow{2}{*}{Type of machine}  &
\begin{tabular}{ccc}
  \multicolumn{3}{c}{Machining time required per unit(minutes)} \\
  \hline
  Component I & Component II & Component III \\
\end{tabular} &
\begin{tabular}{|l}
   Maximum time\\
   available per\\
   week (minutes)\\
\end{tabular} \\
\hline\hline\noalign{\smallskip}
\end{tabular}
\begin{tabular}{p{25mm}p{2cm}p{2cm}p{2cm}p{2cm}}
   Lathes & $a_{11}$ = 12 & $a_{12}$ = 2 & $a_{13}$ = 4 & $b_{1}$ = 1000 \\
   Milling machine & $a_{21}$ = 7 & $a_{22}$ = 5 & $a_{23}$ = 12 & $b_{2}$ = 1500 \\
   Grinding machine & $a_{31}$ = 2 & $a_{32}$ = 4 & $a_{33}$ = 3.5 & $b_{3}$ = 750 \\
   \hline
   Profit per unit & $\; c_{1}$ = 50 & \hspace{0.8mm} $c_{2}$ = 70 & \hspace{1mm} $c_{3}$ = 70 & \\
   \hline\hline
\end{tabular}
\end{scriptsize}
\end{table}
Initially, the deterministic linear programming problem is considered. Thus, the problem takes the form
\begin{equation}\label{excd}
  \begin{array}{c}
    \build{Max}{}{\mathbf{x}} z(x)= 50 x_{1} + 70 x_{2} + 70 x_{3} \\
    \mbox{Subject to:} \\
    12 x_{1} + 2 x_{2} + 4 x_{3} \leq 1000 \\
    7 x_{1} + 5 x_{2} + 12 x_{3} \leq 1500 \\
    2 x_{1} + 4 x_{2} + 3.5 x_{3} \leq 750 \\
    x_{1}, x_{2}, x_{3} \geq 0 \\
  \end{array}
\end{equation}
Table \ref{tab2} shows the optimal solution provided by the \emph{\textbf{nloptr package}} (\citet{j:08}) of free license \emph{\textbf{software R }} (\citet{r:24}):
\begin{table}[h!] 
    \centering 
    \caption{Optimal solution of the deterministic linear programming problem (\ref{excd})} 
    \label{tab2}
    \begin{tabular}{c|c}
      \hline\hline
      Variables & Optimal Solution\\
      \hline
      $x_{1}$ & 47.45763\\
      $x_{2}$ & 123.7288\\
      $x_{3}$ & 45.76271\\
      \hline  
      $z_{max}$ & 14237.2881\\
      \hline\hline
    \end{tabular}
\end{table} 

Now, consider the stochastic linear programming problem, 
\begin{equation}\label{excs}
  \begin{array}{c}
    \build{Max}{}{\mathbf{x}} z(x)= \overline{c}_{1} x_{1} + \overline{c}_{2} x_{2} + \overline{c}_{3} x_{3} \\
    \mbox{Subject to:} \\
    12 x_{1} + 2 x_{2} + 4 x_{3} \leq 1000 \\
    7 x_{1} + 5 x_{2} + 12 x_{3} \leq 1500 \\
    2 x_{1} + 4 x_{2} + 3.5 x_{3} \leq 750 \\
    x_{1}, x_{2}, x_{3} \geq 0 \\
  \end{array}
\end{equation}
where $\mathbf{c}$ is assumed to be a random vector with some elliptically contoured distribution, for which a random sample $\mathbf{c}_{1}, \mathbf{c}_{3}, \dots, \mathbf{c}_{N\mathbf{c}}$ is available. In this case, a random sample of size of $N_{\mathbf{c}} = 12$ describes the profit per unit of each of the machine component I, II and III. In terms of this sample the following results have been calculated  
$$
  \overline{\mathbf{c}} =
    \left(
      \begin{array}{c}
        50 \\
        70 \\
        70 \\
      \end{array}
    \right), \quad
  \mathbf{S}_{\mathbf{c}}^{*} = \frac{1}{(N_{\mathbf{c}}-1)}
     \left(
       \begin{array}{ccc}
         s^{2}_{c_{1}} & 0 & 0\\
         0 & s^{2}_{c_{2}} & 0 \\
         0 & 0 & s^{2}_{c_{3}} \\
       \end{array}
     \right) = 
     \left(
       \begin{array}{ccc}
         450 & 0 & 0\\
         0 & 2600 & 0 \\
         0 & 0 & 850 \\
       \end{array}
     \right)
$$
Thus, the solution of the stochastic linear programming problem of Equation (\ref{excs}) emerges from the following equivalent deterministic nonlinear programming problem:
%\begin{small}
\begin{equation}\label{eqcss}
  \begin{array}{c}
    \build{Max}{}{\mathbf{x}} Z(x)= k_{1}(50 x_{1} + 70 x_{2} + 70 x_{3}) - k_{2} \sqrt{(450 x_{1}^{2} + 2600 x_{2}^{2} + 850 x_{3}^{2})/12 }\\
    \mbox{Subject to:} \\
    12 x_{1} + 2 x_{2} + 4 x_{3} \leq 1000 \\
    7 x_{1} + 5 x_{2} + 12 x_{3} \leq 1500 \\
    2 x_{1} + 4 x_{2} + 3.5 x_{3} \leq 750 \\
    x_{1}, x_{2}, x_{3} \geq 0 \\
  \end{array}
\end{equation}
%\end{small}
Below, three optimal solutions combining the values of $k_{1}$ and $k_{2}$ are displayed.
\begin{table}[h!] 
    \centering 
    \caption{Optimal solutions of the stochastic nonlinear programming problem (\ref{excs})} 
    \label{tab3}
    \begin{scriptsize}
    \begin{tabular}{c|ccc}
      \hline\hline
       \multirow{2}{*}{Variables}& \multicolumn{3}{c}{Optimal solutions} \\
      \cline{2-4}
       & $k_{1} = 0.1$ and $k_{2}=0.9$ & $k_{1} = 0.5$ and $k_{2}=0.5$ & $k_{1} = 0.9$ and $k_{2}=0.1$\\
      \hline
      $x_{1}$ & 51.34569 & 47.45763 & 47.45763 \\
      $x_{2}$ & 28.87406 & 123.7288 & 123.7288 \\
      $x_{3}$ & 81.52589 & 45.76271 & 45.76271 \\
      \hline
      $Z_{max}$ & 249.9528 & 6176.6103 & 12625.1526 \\
      \hline
      $z_{max}$ & 10295.28 & 14237.29 & 14237.29 \\
      \hline\hline
    \end{tabular}
    \end{scriptsize}
\end{table}
\end{example}

\subsection{Case II. The coefficients $\mathbf{a}_{i}, \ i =1,2,\dots,m$, are random.}

In this stage $\mathbf{c}$ and $\mathbf{b}$ are considered fixed vectors. Thus, the chance-constrained programming problem is 
\begin{equation}\label{eqcaip}
\begin{array}{c}
  \build{\min}{}{\mathbf{x}} z(\mathbf{x}) = \mathbf{c}'\mathbf{x} \\[2Ex]
  \mbox{subject to:}\\[2ex] 
  \P(\overline{\mathbf{a}}'_{i}\mathbf{x} \leq b_{i}) \geq (1-\alpha_{i}), \quad i =1,2,\dots,m,  \\[2ex]
  x_{j} \geq 0, \ j =1,2,\dots,n. 
  \end{array}
\end{equation}
Then by Theorem \ref{teo:1}, $\overline{\mathbf{a}}_{i} \sim \mathcal{E}_{n}\left (\boldgreek{\mu}_{\mathbf{a}_{i}}, \frac{1}{N_{\mathbf{a}_{i}}}\mathbf{\Sigma}_{\mathbf{a}_{i}}; g_{\mathbf{a}_{i}}\right )$, $i =1,2,\dots,m$, uncorrelated. 

Define 
$$
  \overline{d}_{i} = \overline{\mathbf{a}}'_{i}\mathbf{x}, \quad i = 1,2,\dots,m,
$$
and denote $\mathbf{\Sigma} = \sigma^{2}$ for $n =1$; therefore, Theorem \ref{teo:1} provides that
$$
  \overline{d}_{i} \sim \mathcal{E}_{1}(\mu_{\overline{d}_{i}}, \frac{1}{N_{\mathbf{a}_{i}}}\sigma^{2}_{\overline{d}_{i}}; g_{\mathbf{a}_{i}}),
$$
where 
$$ 
  \mu_{\overline{d}_{i}} = \boldgreek{\mu}'_{\mathbf{a}_{i}}\mathbf{x}, \ \sigma^{2}_{\overline{d}_{i}}= \mathbf{x}'\mathbf{\Sigma}_{\mathbf{a}_{i}}\mathbf{x}, \mbox{ and } s^{2}_{\overline{d}_{i}} = \mathbf{x}' \mathbf{S}_{\mathbf{a}_{i}}\mathbf{x}\quad i =1,2,\dots,m.
$$ 

Hence, by Theorem \ref{teo3} we have that  
$$
  T_{i} = \frac{\displaystyle\frac{(\overline{d}_{i}-\mu_{\overline{d}_{i}})}{\sqrt{\displaystyle\frac{\sigma^{2}_{\overline{d}_{i}}}{N_{\mathbf{a}_{i}}}}}}
  {\sqrt{\displaystyle\frac{s^{2}_{\overline{d}_{i}}}{(N_{\mathbf{a}_{i}}-1)\sigma^{2}_{\overline{d}_{i}}}}} = \frac{\sqrt{N_{\mathbf{a}_{i}}}\sqrt{N_{\mathbf{a}_{i}}-1}(\overline{d}_{i}-\mu_{\overline{d}_{i}})}{\sqrt{s^{2}_{\overline{d}_{i}}}};
$$
which is a t-student distribution with $N_{\mathbf{a}_{i}}-1$ degrees of freedom for $i =1,2,\dots,m.$; an invariant property associating the same distribution to the random variables $T_{i}$, no matters the elliptically contoured distribution under consideration.

The variables $T_{i}$ can be written in terms of the \textbf{\emph{unbiased sample variances}} $s^{2*} = s^{2}/(N-1)$ or \textbf{\emph{biased sample variances}} $s^{2**}_{N} = s^{2}/N$, as follows
$$
  T_{i} = 
  \left\{
    \begin{array}{ll}
      \displaystyle\frac{\sqrt{N_{\mathbf{a}_{i}}}(\overline{d}_{i}-\mu_{\overline{d}_{i}})}{\sqrt{s^{2*}_{\overline{d}_{i}}}}, &  N_{\mathbf{a}_{i}}-1; \\ \\
      \displaystyle\frac{\sqrt{N_{\mathbf{a}_{i}}-1}(\overline{d}_{i}-\mu_{\overline{d}_{i}})}{\sqrt{s^{2**}_{\overline{d}_{i}}}}, & N_{\mathbf{a}_{i}}.
    \end{array}
  \right.
$$
This way, using the unbiased sample variances, we have that
\begin{eqnarray*}
  \P(\overline{\mathbf{a}}'_{i}\mathbf{x}\leq b_{i})  &=& \P\left(\overline{d}_{i} \leq b_{i}\right) \\
   &=& \P \left( \frac{\displaystyle\frac{(\overline{d}_{i} - \mu_{\overline{d}_{i}})}{\sqrt{\displaystyle\frac{\sigma^{2}_{\overline{d}_{i}}}{N_{\mathbf{a}_{i}}}}}}
  {\sqrt{\displaystyle\frac{s^{2}_{\overline{d}_{i}}}{(N_{\mathbf{a}_{i}}-1)\sigma^{2}_{\overline{d}_{i}}}}} \leq \frac{\displaystyle\frac{(b_{i} - \mu_{\overline{d}_{i}})}{\sqrt{\displaystyle\frac{\sigma^{2}_{\overline{d}_{i}}}{N_{\mathbf{a}_{i}}}}}}
  {\sqrt{\displaystyle\frac{s^{2}_{\overline{d}_{i}}}{(N_{\mathbf{a}_{i}}-1)\sigma^{2}_{\overline{d}_{i}}}}}\right)  \\
   &=& \P \left( T_{i} \leq \frac{\displaystyle\frac{(b_{i} - \mu_{\overline{d}_{i}})}{\sqrt{\displaystyle\frac{\sigma^{2}_{\overline{d}_{i}}}{N_{\mathbf{a}_{i}}}}}}
  {\sqrt{\displaystyle\frac{s^{2}_{\overline{d}_{i}}}{(N_{\mathbf{a}_{i}}-1)\sigma^{2}_{\overline{d}_{i}}}}}\right)\\
   &=&  F_{T_{i}}\left (\frac{\sqrt{N_{\mathbf{a}_{i}}}(b_{i}-\mu_{\overline{d}_{i}})}{\sqrt{s^{2*}_{\overline{d}_{i}}}}\right ), \quad i =1,2,\dots,m,  
\end{eqnarray*}
where $ F_{T_{i}}(\cdot)$ denotes the\textbf{\emph{ distribution function}} of a $t$-student random variable with $N_{\mathbf{a}_{i}}-1$ degrees of freedom. Let $\eta_{i} \in \Re$ such that
$$
  F_{T}(\eta_{i}) = 1-\alpha_{i}, \quad i = 1, 2, \dots, m.
$$ 
then
$$
  F_{T}\left (\frac{\sqrt{N_{\mathbf{a}_{i}}}(b_{i}-\mu_{\overline{d}_{i}})}{\sqrt{s^{2*}_{\overline{d}_{i}}}}\right ) \geq F_{T}(\eta_{i}), \quad i =1,2,\dots,m.
$$
Given that $ F_{T}(\cdot)$ is a non-decreasing monotone function, thus
\begin{equation}\label{eq21}
  \frac{\sqrt{N_{\mathbf{a}_{i}}}(b_{i}-\mu_{\overline{d}_{i}})}{\sqrt{s^{2*}_{\overline{d}_{i}}}} \geq \eta_{i}, \quad i =1,2,\dots,m.
\end{equation}
Since $\mu_{\overline{d}_{i}}$, $i = 1, 2, \dots, m$ are unknown, we can use their unbiased estimators. Then, (\ref{eq21}) can be written as:
$$
 \overline{ \mathbf{a}}'_{i}\mathbf{x}+\frac{\eta_{i}}{\sqrt{N_{\mathbf{a}_{i}}}} \sqrt{s^{2*}_{\overline{d}_{i}}} - b_{i} \leq 0 \quad i =1,2,\dots,m;
$$
or alternatively, it can take the form
\begin{equation}\label{eq22}
 \overline{ \mathbf{a}}'_{i}\mathbf{x}+\frac{\eta_{i}}{\sqrt{N_{\mathbf{a}_{i}}}} \sqrt{\mathbf{x}' \mathbf{S}^{*}_{\mathbf{a}_{i}}\mathbf{x}} 
   - b_{i} \leq 0 \quad i =1,2,\dots,m.
\end{equation}
Thus, the \textbf{\emph{equivalent deterministic nonlinear programming problem to the proposed stochastic linear programming problem}} (\ref{eqcaip}) is given by
\begin{equation}\label{eqaipE}
\begin{array}{c}
  \build{\min}{}{\mathbf{x}} z(\mathbf{x}) = \mathbf{c}'\mathbf{x} \\[2Ex]
  \mbox{subject to:}\\[2ex] 
  \overline{ \mathbf{a}}'_{i}\mathbf{x}+\displaystyle\frac{\eta_{i}}{\sqrt{N_{\mathbf{a}_{i}}}} \sqrt{\mathbf{x}' \mathbf{S}^{*}_{\mathbf{a}_{i}}\mathbf{x}} - b_{i} \leq 0, \quad i =1,2,\dots,m,  \\[2ex]
  x_{j} \geq 0, \ j =1,2,\dots,n, 
  \end{array}
\end{equation}
where $\eta_{i} = F^{-1}_{T_{i}}(1-\alpha_{i})$ is the \textbf{\emph{inverse function}}, often defined as the \textbf{\emph{quantile function}} or \textbf{\emph{percent-point function}}. They correspond to a random variable with a $t$-student distribution with $N_{\mathbf{a}_{i}}-1$ freedom degrees, $i =1,2,\dots,m$.

If the random variables $a_{ij}$ are no correlated, $\mathbf{\Sigma}_{\mathbf{a}_{i}}$, $i = 1,2,\dots,m$ and the corresponding $\mathbf{S}^{*}_{\mathbf{a}_{i}}$, $i = 1,2,\dots,m$ are diagonal matrices as
$$
  \mathbf{\Sigma}_{\mathbf{a}_{i}} =
     \left(
       \begin{array}{cccc}
         \sigma^{2}_{a_{i1}} & 0 & \cdots & 0 \\
         0 & \sigma^{2}_{a_{i2}} & \cdots & 0 \\
         \vdots & \vdots & \ddots & \vdots \\
         0 & 0 & \cdots & \sigma^{2}_{a_{in}} \\
       \end{array}
     \right),
$$
and
$$
     \mathbf{S}_{\mathbf{a}_{i}}^{*} = \frac{1}{(N_{\mathbf{a}_{i}}-1)}
     \left(
       \begin{array}{cccc}
         s^{2}_{a_{i1}} & 0 & \cdots & 0 \\
         0 & s^{2}_{a_{i2}} & \cdots & 0 \\
         \vdots & \vdots & \ddots & \vdots \\
         0 & 0 & \cdots & s^{2}_{a_{in}} \\
       \end{array}
     \right),\ i =1,2,\dots,m,
$$
then the constraints in (\ref{eqaipE}) are reduced to
$$
  \sum_{j=1}^{n}\overline{a}_{ij}x_{j}+\displaystyle\frac{\eta_{i}}{\sqrt{N_{\mathbf{a}_{i}}}} \sqrt{\sum_{j=1}^{n}s^{2*}_{a_{ij}}x_{j}^{2}}
   - b_{i} \leq 0, \quad i =1,2,\dots,m.
$$

\begin{example}\label{ex2}
Consider the Example \ref{ex1}. In this scenario, the machining time required on different machines for each component are not known precisely (as they vary from worker to worker) but are known random samples $\mathbf{a}_{i_{1}}, \dots, \mathbf{a}_{i_{N_{\mathbf{a}_{i}}}},$ $i = 1,2,3$. For this example, $N_{\mathbf{a}_{1}} = N_{\mathbf{a}_{2}} = N_{\mathbf{a}_{3}} = N =25$. These samples are the times that the different workers spend on each step (lathe, milling machine and grinding machine) of the machining of each component I, II and III. The samples provided the following estimators:
%\begin{small}
$$
  \overline{\mathbf{a}}_{1} = 
     \left(
       \begin{array}{c}
         \overline{a}_{11}\\
         \overline{a}_{12}  \\
         \overline{a}_{13}  \\
       \end{array}
     \right) =
     \left(
       \begin{array}{c}
         12 \\
         2 \\
         4 \\
       \end{array}
     \right) \ \ 
   \overline{\mathbf{a}}_{2} = 
     \left(
       \begin{array}{c}
         \overline{a}_{21}\\
         \overline{a}_{22}  \\
         \overline{a}_{23}  \\
       \end{array}
     \right) =
     \left(
       \begin{array}{c}
         7 \\
         5 \\
         12 \\
       \end{array}
     \right) \ \
   \overline{\mathbf{a}}_{3} = 
     \left(
       \begin{array}{c}
         \overline{a}_{31}\\
         \overline{a}_{32}  \\
         \overline{a}_{33}  \\
       \end{array}
     \right) =
     \left(
       \begin{array}{c}
         2 \\
         4 \\
         3.5 \\
       \end{array}
     \right)  
$$
%\end{small} 
and
$$
  \mathbf{S}_{\mathbf{a}_{1}}^{*} = \frac{1}{(N-1)}
     \mathbf{S}_{\mathbf{a}_{1}} = \frac{1}{(N-1)}
     \left(
       \begin{array}{ccc}
         s^{2}_{a_{1_{11}}} & 0 &  0 \\
         0 & s^{2}_{a_{1_{22}}} &  0 \\
         0 & 0 & s^{2}_{a_{1_{33}}} \\
       \end{array}
     \right) =
     \left(
       \begin{array}{ccc}
         30 & 0 &  0 \\
         0 & 10 &  0 \\
         0 & 0 & 12 \\
       \end{array}
     \right),
$$
$$
  \mathbf{S}_{\mathbf{a}_{2}}^{*} = \frac{1}{(N-1)}
     \left(
       \begin{array}{ccc}
         s^{2}_{a_{2_{11}}} & 0 &  0 \\
         0 & s^{2}_{a_{2_{22}}} &  0 \\
         0 & 0 & s^{2}_{a_{2_{33}}} \\
       \end{array}
     \right) =
     \left(
       \begin{array}{ccc}
         22 & 0 &  0 \\
         0 & 32 &  0 \\
         0 & 0 & 15 \\
       \end{array}
     \right),
$$
and
$$
  \mathbf{S}_{\mathbf{a}_{3}}^{*} = \frac{1}{(N-1)}
     \left(
       \begin{array}{ccc}
         s^{2}_{a_{3_{11}}} & 0 &  0 \\
         0 & s^{2}_{a_{3_{22}}} &  0 \\
         0 & 0 & s^{2}_{a_{3_{33}}} \\
       \end{array}
     \right) =
     \left(
       \begin{array}{ccc}
         15 & 0 &  0 \\
         0 & 14 &  0 \\
         0 & 0 & 9 \\
       \end{array}
     \right).
$$
Here we ask for the number of machine components I, II and III that should be manufactured per week to maximize the profit without exceeding the available machining times. This is, the constraints must be met with a probability of at least 0.99.

Then, taking $\alpha_{1} = \alpha_{2} =\alpha_{3} = 0.01$, the stochastic linear programming problem is stated as 
\begin{equation}\label{exais}
  \begin{array}{c}
    \build{Max}{}{\mathbf{x}} z(x)= 50 x_{1} + 70 x_{2} + 70_{3} x_{3} \\
    \mbox{Subject to:} \\
    \P(\overline{a}_{11} x_{1} + \overline{a}_{12} x_{2} + \overline{a}_{13} x_{3} \leq 1000) \geq 0.99 \\
    \P(\overline{a}_{21} x_{1} + \overline{a}_{22} x_{2} + \overline{a}_{23} x_{3} \leq 1500) \geq 0.99 \\
    \P(\overline{a}_{31} x_{1} + \overline{a}_{32} x_{2} + \overline{a}_{33} x_{3} \leq \ 750) \geq 0.99 \\
    x_{1}, x_{2}, x_{3} \geq 0 \\
  \end{array}
\end{equation}
thus, the equivalent deterministic nonlinear programming problem takes de form
\begin{equation}\label{exaiss}
  \begin{array}{c}
    \build{Max}{}{\mathbf{x}} z(x)= 50 x_{1} + 70 x_{2} + 70_{3} x_{3} \\
    \mbox{Subject to:} \\
    12 x_{1} + 2 x_{2} + 4 x_{3} + 2.492159\sqrt{(30 x_{1}^{2} + 10 x_{2}^{2} + 12 x_{3}^{2})/25}\leq 1000 \\
    7 x_{1} + 5 x_{2} + 12 x_{3} + 2.492159\sqrt{(22 x_{1}^{2} + 32 x_{2}^{2} + 15 x_{3}^{2})/25}\leq 1500 \\
    2 x_{1} + 4 x_{2} + 3.5 x_{3} + 2.492159\sqrt{(15 x_{1}^{2} + 14 x_{2}^{2} + 9 x_{3}^{2})/25}\leq 750 \\
    x_{1}, x_{2}, x_{3} \geq 0 \\
  \end{array}
\end{equation}
where $\eta = F^{-1}_{T}(1-0.01) = 2.492159$ is the percent-point of a random variable with a $t$-student distribution with $24$ freedom degrees.

Packages nloptr (\citet{j:08}) of free license software R (\citet{r:24}), provides the following unique optimal solution for the deterministic nonlinear programming problem (\ref{exaiss}), which also corresponds to the optimal solution of the stochastic linear programming (\ref{exais}). The results are summarised in Table \ref{tab4}.

\begin{table}[h!] 
    \centering 
    \caption{Optimal solution of the stochastic linear programming problem (\ref{exais})} 
    \label{tab4}
    \begin{tabular}{c|c}
      \hline\hline
      Variables & Optimal Solution\\
      \hline
      $x_{1}$ & 38.84635\\
      $x_{2}$ & 81.64707\\
      $x_{3}$ & 46.38850\\
      \hline  
      $z_{max}$ & 10904.8076\\
      \hline\hline
    \end{tabular}
\end{table} 
\end{example}

\subsection{Case III. The coefficients $\mathbf{b}$ are random.}

Now, $\mathbf{c}$ and $\mathbf{a}_{i}$, $i = 1,2,\dots,m,$ are considered fixed vectors. In this scenario, the chance-constrained programming problem is given by 
\begin{equation}\label{eqbp}
\begin{array}{c}
  \build{\min}{}{\mathbf{x}} z(\mathbf{x}) = \mathbf{c}'\mathbf{x} \\[2Ex]
  \mbox{subject to:}\\[2ex] 
  \P(\mathbf{a}'_{i}\mathbf{x} \leq \overline{b}_{i}) \geq (1-\alpha_{i}), \quad i =1,2,\dots,m,  \\[2ex]
  x_{j} \geq 0, \ j =1,2,\dots,n. 
  \end{array}
\end{equation}
In addition, by Theorem \ref{teo:1}, $\overline{\mathbf{b}} = (\overline{b}_{1} \ \overline{b}_{2}\ \cdots \  \overline{b}_{m})' \sim \mathcal{E}_{m}\left (\boldgreek{\mu}_{\mathbf{b}}, \frac{1}{N_{\mathbf{b}}}\mathbf{\Sigma}_{\mathbf{b}}; g_{\mathbf{b}}\right )$ and 
$$
  \overline{b}_{i} \sim \mathcal{E}_{1}\left(\mu_{b_{i}}, \frac{1}{N_{\mathbf{b}}}\sigma^{2}_{b_{i}}; g_{\mathbf{b}}\right ),
$$
where 
$$
  \mathbf{\mu}_{\mathbf{b}} =
       \left (
          \begin{array}{c}
            \mu_{b_{1}} \\
            \mu_{b_{2}} \\
            \vdots \\
            \mu_{b_{m}} \\ 
          \end{array}
       \right ),  
       \quad 
       \mathbf{\Sigma}_{\mathbf{b}} = 
      \left\{
        \begin{array}{ll}
          \sigma^{2}_{b_{i}}, & \hbox{$i =j$;} \\
          \sigma_{_{b_{i}b_{j}}}, & \hbox{$i\neq j$.}
        \end{array}
      \right. , \mbox{ and } 
       \mathbf{S}_{\mathbf{b}} = 
      \left\{
        \begin{array}{ll}
          s^{2}_{b_{i}}, & \hbox{$i =j$;} \\
          s_{_{b_{i}b_{j}}}, & \hbox{$i\neq j$.}
        \end{array}
      \right.
$$
Therefore, for $i =1,2,\dots,m$, Theorem \ref{teo3} states that
$$
  T_{i} = \frac{\displaystyle\frac{(\overline{b}_{i}-\mu_{\overline{b}_{i}})}{\sqrt{\displaystyle\frac{\sigma^{2}_{\overline{b}_{i}}}{N_{\mathbf{b}}}}}}
  {\sqrt{\displaystyle\frac{s^{2}_{\overline{b}_{i}}}{(N_{\mathbf{b}}-1)\sigma^{2}_{\overline{b}_{i}}}}} = 
  \frac{\sqrt{N_{\mathbf{b}}}\sqrt{N_{\mathbf{b}}-1}(\overline{b}_{i}-\mu_{\overline{b}_{i}})}{\sqrt{s^{2}_{\overline{b}_{i}}}}= 
  \frac{\sqrt{N_{\mathbf{b}}}(\overline{b}_{i}-\mu_{\overline{b}_{i}})}{\sqrt{s^{2*}_{\overline{b}_{i}}}}
$$
follows a $t$-student distribution with $N_{\mathbf{b}}-1$ degrees of freedom, which is invariant under any elliptically contoured distribution assumed for the random vector $\mathbf{b}$.

The constraints in (\ref{eqbp}) provide that
\begin{eqnarray}
  \P(\mathbf{a}'_{i}\mathbf{x}\leq \overline{b}_{i})  &=&  \P \left( \frac{\displaystyle\frac{(\mathbf{a}'_{i}\mathbf{x}-\mu_{\overline{b}_{i}})}{\sqrt{\displaystyle\frac{\sigma^{2}_{\overline{b}_{i}}}{N_{\mathbf{b}}}}}}
  {\sqrt{\displaystyle\frac{s^{2}_{\overline{b}_{i}}}{(N_{\mathbf{b}}-1)\sigma^{2}_{\overline{b}_{i}}}}} \leq \frac{\displaystyle\frac{(\overline{b}_{i}-\mu_{\overline{b}_{i}})}{\sqrt{\displaystyle\frac{\sigma^{2}_{\overline{b}_{i}}}{N_{\mathbf{b}}}}}}
  {\sqrt{\displaystyle\frac{s^{2}_{\overline{b}_{i}}}{(N_{\mathbf{b}}-1)\sigma^{2}_{\overline{b}_{i}}}}}\right) \nonumber \\
   &=& \P \left( \frac{\displaystyle\frac{(\mathbf{a}'_{i}\mathbf{x} -\mu_{\overline{b}_{i}})}{\sqrt{\displaystyle\frac{\sigma^{2}_{\overline{b}_{i}}}{N_{\mathbf{b}}}}}}
  {\sqrt{\displaystyle\frac{s^{2}_{\overline{b}_{i}}}{(N_{\mathbf{b}}-1)\sigma^{2}_{\overline{b}_{i}}}}}  \leq T_{i} \right) \nonumber\\
  &=& 1 - \P \left( T_{i}  \leq \frac{\displaystyle\frac{(\mathbf{a}'_{i}\mathbf{x} -\mu_{\overline{b}_{i}})}{\sqrt{\displaystyle\frac{\sigma^{2}_{\overline{b}_{i}}}{N_{\mathbf{b}}}}}}
  {\sqrt{\displaystyle\frac{s^{2}_{\overline{b}_{i}}}{(N_{\mathbf{b}}-1)\sigma^{2}_{\overline{b}_{i}}}}}  \right) \geq 1- \alpha_{i} \nonumber\\
%\end{eqnarray*}
%\begin{eqnarray}
% \phantom{\P(\mathbf{a}'_{i}\mathbf{x}\leq \overline{b}_{i})} 
  &=& \P \left( T_{i}  \leq \frac{\displaystyle\frac{(\mathbf{a}'_{i}\mathbf{x} -\mu_{\overline{b}_{i}})}{\sqrt{\displaystyle\frac{\sigma^{2}_{\overline{b}_{i}}}{N_{\mathbf{b}}}}}}
  {\sqrt{\displaystyle\frac{s^{2}_{\overline{b}_{i}}}{(N_{\mathbf{b}}-1)\sigma^{2}_{\overline{b}_{i}}}}}  \right) \leq \alpha_{i} \nonumber\\ \label{resb}
   &=& F_{T_{i}}\left (\frac{\sqrt{N_{\mathbf{b}}}(\mathbf{a}'_{i}\mathbf{x}-\mu_{\overline{b}_{i}})}{\sqrt{s^{2*}_{\overline{b}_{i}}}}\right ) \leq \alpha_{i}, \quad i =1,2,\dots,m,  
\end{eqnarray}
where $ F_{T_{i}}(\cdot)$ denotes the\textbf{\emph{ distribution function}} of a random variable with a $t$-student distribution and $N_{\mathbf{b}}-1$ freedom degrees. Now, if $\delta_{i}$ is the percentile of the $t$-student variate with $N_{\mathbf{b}}-1$ freedom degrees such that
$$
   F_{T_{i}}\left (\delta_{i}\right) = \alpha_{i}, \quad i =1,2,\dots,m,
$$
then, the constraints in (\ref{resb}) can be expressed as
$$
  F_{T_{i}}\left (\frac{\sqrt{N_{\mathbf{b}}}(\mathbf{a}'_{i}\mathbf{x}-\mu_{\overline{b}_{i}})}{\sqrt{s^{2*}_{\overline{b}_{i}}}}\right ) \leq F_{T_{i}}\left (\delta_{i}\right), \quad i =1,2,\dots,m.
$$
Given that $ F_{T}(\cdot)$ is a non-decreasing monotonic function, then
$$
  \frac{\sqrt{N_{\mathbf{b}}}(\mathbf{a}'_{i}\mathbf{x}-\mu_{\overline{b}_{i}})}{\sqrt{s^{2*}_{\overline{b}_{i}}}} \leq \delta_{i}, \quad i =1,2,\dots,m,
$$
or
$$
  \mathbf{a}'_{i}\mathbf{x}-\mu_{\overline{b}_{i}} - \frac{\delta_{i}}{\sqrt{N_{\mathbf{b}}}}\sqrt{s^{2*}_{\overline{b}_{i}}} \leq 0, , \quad i =1,2,\dots,m.
$$
Note that $\mu_{\overline{b}_{i}}$, $i =1,2,\dots,m$ are unknow and can be replaced by their corresponding unbiased estimators $\overline{\mathbf{b}}$. Thus, \textbf{\emph{the stochastic linear programming problem stated in (\ref{eqbp}) is equivalent to the following deterministic linear programming problem}}:
\begin{equation}\label{eqbpE}
\begin{array}{c}
  \build{\min}{}{\mathbf{x}} z(\mathbf{x}) = \mathbf{c}'\mathbf{x} \\[2Ex]
  \mbox{subject to:}\\[2ex] 
  \mathbf{a}'_{i}\mathbf{x}-\overline{b}_{i} - \displaystyle\frac{\delta_{i}}{\sqrt{N_{\mathbf{b}}}}\sqrt{s^{2*}_{\overline{b}_{i}}} \leq 0, , \quad i =1,2,\dots,m,  \\[2ex]
  x_{j} \geq 0, \ j =1,2,\dots,n. 
  \end{array}
\end{equation}

\begin{example}\label{ex3}
Under the conditions of Example \ref{ex1}, assume that the machining times available on different machines are stochastic (these machining times depend on the time spent on maintenance and repair); specifically, consider a random vector $\mathbf{b}$ following an arbitrary elliptically contoured distribution. Furthermore, there is a random sample $\mathbf{b}_{1}, \mathbf{b}_{2}, \dots, \mathbf{b}_{N_{\mathbf{b}}}$, $N_{\mathbf{b}} = 25$ characterized by the following estimators:
\begin{small} 
$$
  \overline{\mathbf{b}} =
    \left(
      \begin{array}{c}
        1000 \\
        1500 \\
        750 \\
      \end{array}
    \right), \quad
  \mathbf{S}_{\mathbf{b}}^{*} = \frac{1}{(N_{\mathbf{b}}-1)}
     \left(
       \begin{array}{ccc}
         s^{2}_{b_{1}} & 0 & 0\\
         0 & s^{2}_{b_{2}} & 0 \\
         0 & 0 & s^{2}_{b_{3}} \\
       \end{array}
     \right) = 
     \left(
       \begin{array}{ccc}
         5000 & 0 & 0\\
         0 & 4000 & 0 \\
         0 & 0 & 500 \\
       \end{array}
     \right).
$$
\end{small}
We are interested in the number of machine components I, II and III that must be manufactured per week in order to maximize the profit. The constraints must be met with a probability of at least 0.99.

For this case, the stochastic linear programming problem is stated as:
\begin{equation}\label{exbs}
  \begin{array}{c}
    \build{Max}{}{\mathbf{x}} z(x)= 50 x_{1} + 70 x_{2} + 70 x_{3} \\
    \mbox{Subject to:} \\
    \P(12 x_{1} + 2 x_{2} + 4 x_{3} \leq \overline{b}_{1}) \geq 0.99 \\
    \P(7 x_{1} + 5 x_{2} + 12 x_{3} \leq \overline{b}_{2}) \geq 0.99 \\
    \P(2 x_{1} + 4 x_{2} + 3.5 x_{3} \leq \overline{b}_{3}) \geq 0.99 \\
    x_{1}, x_{2}, x_{3} \geq 0. \\
  \end{array}
\end{equation}
Hence, the equivalent deterministic linear programming problem is
\begin{equation}\label{exbss}
  \begin{array}{c}
    \build{Max}{}{\mathbf{x}} z(x)= 50 x_{1} + 70 x_{2} + 70_{3} x_{3} \\
    \mbox{Subject to:} \\
    12 x_{1} + 2 x_{2} + 4 x_{3} - (-2.492159)\sqrt{5000/25}\leq 1000 \\
    7 x_{1} + 5 x_{2} + 12 x_{3} - (-2.492159)\sqrt{4000/25}\leq 1500 \\
    2 x_{1} + 4 x_{2} + 3.5 x_{3} - (-2.492159)\sqrt{500/25}\leq 750 \\
    x_{1}, x_{2}, x_{3} \geq 0, \\
  \end{array}
\end{equation}
where $\delta = F^{-1}_{T}(0.01) = -2.492159$ is the percent-point of a $t$-student random variable with $24$ freedom degrees.
The application of nloptr package (\citet{j:08}) of software R (\citet{r:24}) achieves the following unique optimal solution of the stochastic linear programming problem (\ref{exbs}) 
\begin{table}[h!] 
    \centering 
    \caption{Optimal solution of the stochastic linear programming problem (\ref{exbs})} 
    \label{tab5}
    \begin{tabular}{c|c}
      \hline\hline
      Variables & Optimal Solution\\
      \hline
      $x_{1}$ & 44.92657\\
      $x_{2}$ & 122.91980\\
      $x_{3}$ & 44.94930\\
      \hline  
      $z_{max}$ & 13997.1624\\
      \hline\hline
    \end{tabular}
\end{table} 
\end{example}

\subsection{Case IV. The coefficients $\mathbf{c}$, $\mathbf{a}_{i}, \ i =1,2,\dots,m.$ and $\mathbf{b}$ are random.}

In this case, the stochastic linear programming problem is stated as
\begin{equation}\label{eqcaibp}
\begin{array}{c}
  \build{\min}{}{\mathbf{x}} z(\mathbf{x}) = \overline{\mathbf{c}}'\mathbf{x} \\[2Ex]
  \mbox{subject to:}\\[2ex] 
  \P(\overline{\mathbf{a}}'_{i}\mathbf{x} \leq \overline{b}_{i}) \geq (1-\alpha_{i}), \quad i =1,2,\dots,m,  \\[2ex]
  x_{j} \geq 0, \ j =1,2,\dots,n. 
  \end{array}
\end{equation}
The objective function involves only the random vector $\mathbf{c}$, then we can consider the same cost $Z(\mathbf{x})$ of (\ref{eqcpE}). The remaining randomness is arranged in the samples $\mathbf{g}_{i1}, \mathbf{g}_{i2}, \dots, \mathbf{g}_{iN}$, $i = 1,2,\dots,m$, such that
$$
  \mathbf{g}_{ij} =
    \left(
      \begin{array}{c}
        \mathbf{a}_{ik} \\
        b_{ik} \\
      \end{array}
    \right), i = 1,2,\dots,m, \ k = 1,2,\dots,N,
$$
where
$$
  \mathbf{G}_{i} = \left(
     \begin{array}{c}
        \mathbf{g}'_{i_{1}} \\
        \mathbf{g}'_{i_{2}} \\
        \vdots \\
        \mathbf{g}'_{i_{N}} \\
      \end{array}
      \right) \sim \mathcal{E}_{N \times n} (\mathbf{1}_{N}\boldgreek{\mu}'_{\mathbf{g}_{i}}, \mathbf{I}_{N} 
       \otimes \mathbf{\Sigma}_{\mathbf{g}_{i}}; g_{\mathbf{g}_{i}}), \ i = 1, 2, \dots, m,
$$ 
and for $i =1,2,\dots,m$, we have the estimators
$$
  \overline{\mathbf{g}}_{i} = \frac{1}{N} \displaystyle\sum_{k =1}^{N} \mathbf{g}_{i_{k}} = 
    \left (
        \begin{array}{c}
          \overline{\mathbf{a}}_{i} \\
         \overline{b}_{i} 
        \end{array}
     \right ),
$$
and
$$ 
  \mathbf{S}_{\mathbf{g}_{i}} = \displaystyle\sum_{k =1}^{N} (\mathbf{g}_{i_{k}} -\overline{\mathbf{g}}_{i})(\mathbf{g}_{i_{k}} - 
  \overline{\mathbf{g}}_{i})'  = \mathbf{G}'_{i}\left(\mathbf{I}_{N} - \frac{1}{N} \mathbf{1}_{N} \mathbf{1}'_{N}\right)\mathbf{G}_{i}. 
$$
Recall that $\mathbf{S}^{*}_{\mathbf{g}_{i}} = (N-1)^{-1} \mathbf{S}_{\mathbf{g}_{i}}$. 
Therefore by Theorem \ref{teo:1}, 
$$
  \overline{\mathbf{g}}_{i} \sim \mathcal{E}_{n+1}(\boldgreek{\mu}_{\mathbf{g}_{i}},\frac{1}{N}\mathbf{\Sigma}_{\mathbf{g}_{i}}, g_{\mathbf{g}_{i}}),
$$
where
$$
  \E(\overline{\mathbf{g}}_{i}) = \boldgreek{\mu}_{\mathbf{g}_{i}} = 
     \left (
        \begin{array}{c}
          \boldgreek{\mu}_{\mathbf{a}_{i}} \\
         \mu_{{b}_{i}} 
        \end{array}
     \right ),\quad i =1,2,\dots,m,
$$
$$
  \mathbf{\Sigma}_{\mathbf{g}_{i}} = 
     \left (
        \begin{array}{cc}
          \mathbf{\Sigma}_{\mathbf{a}_{i}} & \mathbf{v}_{i} \\
          \mathbf{v}'_{i} & \sigma^{2}_{b_{i}} 
        \end{array}
     \right ), \quad i =1,2,\dots,m,
$$
and 
$$
  \mathbf{v}_{i} \propto 
     \left (
        \begin{array}{c}
          \Cov(\overline{a}_{i1},\overline{b}_{1})\\
          \Cov(\overline{a}_{i2},\overline{b}_{2}) \\
          \vdots \\
          \Cov(\overline{a}_{in},\overline{b}_{n}) 
        \end{array}
     \right ), \quad i =1,2,\dots,m.
$$

Note that this stage assumes $N_{\mathbf{a}_{i}} = N_{\mathbf{b}} = N$, and $g_{\mathbf{g}_{i}}$ as the joint kernel of $\mathbf{g}_{i}$. 

Thus, the constraints in (\ref{eqcaibp}) take the form
$$
  \P\left (\overline{q}_{i}(\mathbf{x}) \leq 0 \right ) \geq 1-\alpha_{i}, \quad i =1,2,\dots,m,
$$
where 
$$
  \overline{q}_{i}(\mathbf{x}) = \overline{\mathbf{a}}'_{i}\mathbf{x} - \overline{b}_{i} = \overline{\mathbf{g}}'_{i}\mathbf{y}, \quad i =1,2,\dots,m,
$$
and 
$$
  \mathbf{y} =
    \left (
        \begin{array}{c}
          \mathbf{x} \\
          -1 
        \end{array}
     \right ) \in \Re^{n+1}.
$$ 

Observe that $ \overline{q}_{i}(\mathbf{x})$, $i =1,2,\dots,m$, are given by a linear combination of elliptically contoured distributed random vectors $\mathbf{g}_{i}$, then Theorem \ref{teo:1} also attributes them an elliptically contoured distribution. The mean and variance of $ \overline{q}_{i}(\mathbf{x})$ are given by
$$
  \boldgreek{\mu}_{\overline{q}_{i}(\mathbf{x})} = \boldgreek{\mu}'_{\mathbf{g}_{i}}\mathbf{y}, \quad i =1,2,\dots,m,   
$$
and 
$$
  \sigma^{2}_{\overline{q}_{i}(\mathbf{x})} = \frac{1}{N}\mathbf{y}'\mathbf{\Sigma}_{\mathbf{g}_{i}}\mathbf{y}, \quad i =1,2,\dots,m,
$$ 
Then, $\overline{q}_{i}(\mathbf{x}) \sim \mathcal{E}_{n}(\boldgreek{\mu}'_{\mathbf{g}_{i}}\mathbf{y},\frac{1}{N}\mathbf{y}'\mathbf{\Sigma}_{\mathbf{g}_{i}}\mathbf{y},g_{\mathbf{g}_{i}})$, 
 and the constraints (\ref{eqcaibp}) can be restated as
\begin{eqnarray*}
  \P\left (\overline{q}_{i}(\mathbf{x}) \leq 0 \right )  &=& \P \left( \frac{\displaystyle\frac{(\overline{q}_{i}(\mathbf{x}) - \boldgreek{\mu}_{\overline{q}_{i}(\mathbf{x})})}{\sqrt{\displaystyle\frac{\sigma^{2}_{\overline{q}_{i}(\mathbf{x})}}{N}}}}
  {\sqrt{\displaystyle\frac{s^{2}_{\overline{q}_{i}(\mathbf{x})}}{(N-1)\sigma^{2}_{\overline{q}_{i}(\mathbf{x})}}}} \leq  \frac{\displaystyle\frac{ - \boldgreek{\mu}_{\overline{q}_{i}(\mathbf{x})}}{\sqrt{\displaystyle\frac{\sigma^{2}_{\overline{q}_{i}(\mathbf{x})}}{N}}}}
  {\sqrt{\displaystyle\frac{s^{2}_{\overline{q}_{i}(\mathbf{x})}}{(N-1)\sigma^{2}_{\overline{q}_{i}(\mathbf{x})}}}}\right)  \\
   &=& \P \left( T_{i} \leq\frac{\displaystyle\frac{ - \boldgreek{\mu}_{\overline{q}_{i}(\mathbf{x})}}{\sqrt{\displaystyle\frac{\sigma^{2}_{\overline{q}_{i}(\mathbf{x})}}{N}}}}
  {\sqrt{\displaystyle\frac{s^{2}_{\overline{q}_{i}(\mathbf{x})}}{(N-1)\sigma^{2}_{\overline{q}_{i}(\mathbf{x})}}}}\right) \\
   &=&  F_{T_{i}}\left (\frac{- \sqrt{N} \ \mu_{\overline{q}_{i}(\mathbf{x})}}{\sqrt{s^{2*}_{\overline{q}_{i}(\mathbf{x})}}}\right ), \quad i =1,2,\dots,m,  
\end{eqnarray*}
where $s^{2*}_{\overline{q}_{i}(\mathbf{x})} = \frac{1}{N-1}\mathbf{y}'\mathbf{S}^{*}_{\mathbf{g}_{i}}\mathbf{y}$, $i =1,2,\dots,m,$ and $ F_{T_{i}}(\cdot)$ denotes the $t$-student distribution with $N-1$ freedom degrees. Denote $\tau_{i}$ as the percentile of the $t$-student variate with $N-1$ freedom degrees, such that
$$
   F_{T_{i}}\left (\tau_{i}\right) = 1 - \alpha_{i}, \quad i =1,2,\dots,m,
$$
then, the constraints (\ref{resb}) are given by
$$
  F_{T_{i}}\left (\frac{- \sqrt{N} \ \mu_{\overline{q}_{i}(\mathbf{x})}}{\sqrt{s^{2*}_{\overline{q}_{i}(\mathbf{x})}}}\right ) \geq F_{T}(\tau_{i}), \quad i =1,2,\dots,m,
$$
Taking into account that $ F_{T}(\cdot)$ is a non-decreasing monotonic function, these inequalities shall be satisfied only if the following deterministic nonlinear inequalities hold
\begin{equation}\label{eq23}
   \frac{- \sqrt{N} \ \mu_{\overline{q}_{i}(\mathbf{x})}}{\sqrt{s^{2*}_{\overline{q}_{i}(\mathbf{x})}}} \geq \tau_{i}, \quad i =1,2,\dots,m,
\end{equation}
Since $\mu_{\overline{q}_{i}(\mathbf{x})}$, $i = 1, 2, \dots, m$, are unknown, their unbiased estimator can be used instead. Finally, (\ref{eq23}) can be expressed as 
$$
 \overline{q}_{i}(\mathbf{x}) + \frac{\tau_{i}}{\sqrt{N}} \sqrt{s^{2*}_{\overline{q}_{i}(\mathbf{x})}} \leq 0, \quad i =1,2,\dots,m,
$$
Hence, \textbf{\emph{the stochastic linear programming problem  (\ref{eqcaibp}) can be stated as the following equivalent  deterministic nonlinear programming problem}}:
\begin{equation}\label{eqcaibpE}
\begin{array}{c}
  \build{\min}{}{\mathbf{x}} Z(\mathbf{x}) = k_{1}\overline{\mathbf{c}}'\mathbf{x} + k_{2} \displaystyle\sqrt{ \widetilde{\Var}(\overline{\mathbf{c}}'\mathbf{x})}\\[2Ex]
  \mbox{subject to:}\\[2ex] 
  \overline{q}_{i}(\mathbf{x}) + \displaystyle\frac{\tau_{i}}{\sqrt{N}} \sqrt{s^{2*}_{\overline{q}_{i}(\mathbf{x})}} \leq 0, \quad i =1,2,\dots,m,  \\[2ex]
  x_{j} \geq 0, \ j =1,2,\dots,n. 
  \end{array}
\end{equation}

\begin{example}\label{ex4}
Consider the extreme random scenario where the machining times required, the maximum time available and the unit profits are all assumed to be elliptically contoured distributed. Assume that there are random samples $\mathbf{g}_{i1}, \mathbf{g}_{i2}, \dots, \mathbf{g}_{iN}$, $i = 1,2,\dots,m$ described by the following estimators 
$$
  \overline{\mathbf{c}} =
    \left(
      \begin{array}{c}
        50 \\
        70 \\
        70 \\
      \end{array}
    \right), \quad
  \mathbf{S}_{\mathbf{c}}^{*} = \frac{1}{(N_{\mathbf{c}}-1)}
     \left(
       \begin{array}{ccc}
         s^{2}_{c_{1}} & 0 & 0\\
         0 & s^{2}_{c_{2}} & 0 \\
         0 & 0 & s^{2}_{c_{3}} \\
       \end{array}
     \right) = 
     \left(
       \begin{array}{ccc}
         450 & 0 & 0\\
         0 & 2600 & 0 \\
         0 & 0 & 850 \\
       \end{array}
     \right),
$$
where $N_{\mathbf{c}} = 12$, $\overline{\mathbf{g}}_{1} = (\overline{\mathbf{a}}'_{i}, b_{i})'$
$$
  \overline{\mathbf{g}}_{1} = 
     \left(
       \begin{array}{c}
         \overline{a}_{11}\\
         \overline{a}_{12}\\
         \overline{a}_{13}\\
         \overline{b}_{1}\\
       \end{array}
     \right) =
     \left(
       \begin{array}{c}
         12 \\
         2 \\
         4 \\
         1000\\
       \end{array}
     \right) \ \ 
   \overline{\mathbf{g}}_{2} = 
     \left(
       \begin{array}{c}
         \overline{a}_{21}\\
         \overline{a}_{22}\\
         \overline{a}_{23}\\
         \overline{b}_{2}\\
       \end{array}
     \right) =
     \left(
       \begin{array}{c}
         7 \\
         5 \\
         12 \\
         1500\\
       \end{array}
     \right)
$$
$$
   \overline{\mathbf{g}}_{3} = 
     \left(
       \begin{array}{c}
         \overline{a}_{31}\\
         \overline{a}_{32}\\
         \overline{a}_{33}\\
         \overline{b}_{3}\\
       \end{array}
     \right) =
     \left(
       \begin{array}{c}
         2 \\
         4 \\
         3.5\\
         750\\
       \end{array}
     \right),  
$$ 
and
$$
  \mathbf{S}_{\mathbf{g}_{1}}^{*} =  \frac{1}{(N-1)}
     \left(
       \begin{array}{cccc}
         s^{2}_{a_{1_{11}}} & 0 &  0 & 0\\
         0 & s^{2}_{a_{1_{22}}} &  0 & 0\\
         0 & 0 & s^{2}_{a_{1_{33}}} & 0\\
         0 & 0 & 0 & s^{2}_{b_{1}}\\
       \end{array}
     \right) =
     \left(
       \begin{array}{cccc}
         30 & 0 & 0  & 0\\
         0 & 10 & 0  & 0\\
         0 & 0 & 12  & 0\\
         0 & 0 & 0  & 5000\\
       \end{array}
     \right),
$$
$$
  \mathbf{S}_{\mathbf{g}_{2}}^{*} = \frac{1}{(N-1)}
     \left(
       \begin{array}{cccc}
         s^{2}_{a_{2_{11}}} & 0 &  0 \\
         0 & s^{2}_{a_{2_{22}}} &  0 \\
         0 & 0 & s^{2}_{a_{2_{33}}} & 0 \\
         0 & 0 & 0 & s^{2}_{b_{2}} \\
       \end{array}
     \right) =
     \left(
       \begin{array}{cccc}
         22 & 0 &  0 & 0\\
         0 & 32 &  0 & 0\\
         0 & 0 & 15 & 0\\
         0 & 0 & 0 & 4000\\
       \end{array}
     \right),
$$
and
$$
  \mathbf{S}_{\mathbf{g}_{3}}^{*} = \frac{1}{(N-1)}
     \left(
       \begin{array}{cccc}
         s^{2}_{a_{3_{11}}} & 0 &  0  & 0\\
         0 & s^{2}_{a_{3_{22}}} &  0  & 0\\
         0 & 0 & s^{2}_{a_{3_{33}}}  & 0\\
         0 & 0 & 0  & s^{2}_{b_{3}}\\
       \end{array}
     \right) =
     \left(
       \begin{array}{cccc}
         15 & 0 & 0  & 0\\
         0 & 14 & 0  & 0\\
         0 & 0 & 9  & 0\\
         0 & 0 & 0  & 500\\
       \end{array}
     \right),
$$
with $N = 25$.

We ask for the number of components I, II y III to be manufactured per week in order to maximize the profit. The constraints have to be satisfied with a probability of at leas 0.99.

In this setting, the stochastic linear programming problem is stated as follows:
\begin{equation}\label{excaibs}
  \begin{array}{c}
    \build{Max}{}{\mathbf{x}} z(x)= \overline{c}_{1} x_{1} + \overline{c}_{2} x_{2} + \overline{c}_{3} x_{3} \\
    \mbox{Subject to:} \\
    \P(\overline{a}_{11} x_{1} + \overline{a}_{12} x_{2} + \overline{a}_{13} x_{3} \leq \overline{b}_{1}) \geq 0.99 \\
    \P(\overline{a}_{21} x_{1} + \overline{a}_{22} x_{2} + \overline{a}_{23} x_{3} \leq \overline{b}_{2}) \geq 0.99 \\
    \P(\overline{a}_{31} x_{1} + \overline{a}_{32} x_{2} + \overline{a}_{33} x_{3} \leq \overline{b}_{3}) \geq 0.99 \\
    x_{1}, x_{2}, x_{3} \geq 0. \\
  \end{array}
\end{equation}
Then, the equivalent deterministic nonlinear programming problem takes the form
\begin{equation}\label{excaibss}
  \begin{array}{c}
    \build{Max}{}{\mathbf{x}} Z(x)= k_{1}(50 x_{1} + 70 x_{2} + 70 x_{3}) - k_{2} \sqrt{(450 x_{1}^{2} + 2600 x_{2}^{2} + 850 x_{3}^{2})/12 }\\
    \mbox{Subject to:} \\
    12 x_{1} + 2 x_{2} + 4 x_{3} + 2.492159\sqrt{(30 x_{1}^{2} + 10 x_{2}^{2} + 12 x_{3}^{2}+5000)/25}\leq 1000 \\
    7 x_{1} + 5 x_{2} + 12 x_{3} + 2.492159\sqrt{(22 x_{1}^{2} + 32 x_{2}^{2} + 15 x_{3}^{2}+4000)/25}\leq 1500 \\
    2 x_{1} + 4 x_{2} + 3.5 x_{3} + 2.492159\sqrt{(15 x_{1}^{2} + 14 x_{2}^{2} + 9 x_{3}^{2}+500)/25}\leq 750 \\
    x_{1}, x_{2}, x_{3} \geq 0, \\
  \end{array}
\end{equation}
where $\tau = F^{-1}_{T}(0.99) = 2.492159$ is the percent-point of a random variable with a $t$-student distribution of $24$ freedom degrees.

Finally, Table \ref{tab6} shows the optimal solutions for the stochastic linear programming problem (\ref{excaibs}) with different values of $k_{1}$ and $k_{2}$ obtained with the soport of packages nloptr (\citet{j:08}) of free license software R (\citet{r:24}). 

\begin{table}[!ht] 
    \centering 
    \caption{Optimal solutions of the stochastic nonlinear programming problem (\ref{excs})} 
    \label{tab6}
    \begin{scriptsize}
    \begin{tabular}{c|ccc}
      \hline\hline
       \multirow{2}{*}{Variables}& \multicolumn{3}{c}{Optimal solutions} \\
      \cline{2-4}
       & $k_{1} = 0.25$ and $k_{2}=0.75$ & $k_{1} = 0.5$ and $k_{2}=0.5$ & $k_{1} = 0.75$ and $k_{2}=0.25$\\
      \hline
      $x_{1}$ & 38.98367 & 38.59630 & 38.59630 \\
      $x_{2}$ & 60.98913 & 81.75362 & 81.75362 \\
      $x_{3}$ & 57.95659 & 46.33119 & 46.33119 \\
      \hline
      $Z_{max}$ & 1781.9370 & 4804.4404 & 7850.0963 \\
      \hline
      $z_{max}$ & 10275.38  & 10895.75   & 10895.75 \\
      \hline\hline
    \end{tabular}
    \end{scriptsize}
\end{table}
\end{example}

\section*{Conclusions}

A generalisation of the choice-constraints Charnes-Cooper approach is presented. This new theory replaces the usual literature normality assumption by a robust elliptically contoured distribution. The technique also accepts a different elliptical distribution for each parameter in the stochastic linear programming problem. Moreover, instead of assuming known distribution parameters, the theorems estimate the parameters of the stochastic linear programming problem by available random samples of plausible different sizes. The new stochastic optimisation achieves a modification of the chance-constrained algorithm with an invariance under the choice of any assumed elliptically contoured distribution. This situation is more closely related to practical problems in different areas of knowledge.

Although the extreme simultaneous randomness setting was derived for $\mathbf{c}$, $\mathbf{a}_{i}, \ i =1,2,\dots,m.$ and $\mathbf{b}$ $\mathbf{a}_{i}$, $i = 1,2,\dots,m$ and $\mathbf{b}$, the remaining three pairwise combinations for simultaneous uncertainty can be achieved easily by muting the third deterministic variable in the main result of Case IV.  Thus, the equivalent deterministic (linear or nonlinear) programming problems can be solved under elliptical invariance for random $\mathbf{a}_{i}$, $i = 1,2,\dots,m$ and $\mathbf{b}$; $\mathbf{a}_{i}$, $i = 1,2,\dots,m$ and $\mathbf{c}$; and,  $\mathbf{b}$ and $\mathbf{c}$.

The invariance property enlarges the bounder for a demanding future research with similar advantage but under multi modal classes of distributions, real normed division algebras and/or shape artificial intelligence.

%\section*{Acknowledgements}

%The authors wish to thank the Editor and the anonymous reviewers for their constructive comments
%on the preliminary version of this paper.

 \end{document}